\documentclass[a4paper,reqno,10pt]{amsart}

\usepackage[USenglish]{babel}

\usepackage{amsmath}
\usepackage{amssymb}
\usepackage{verbatim}
\usepackage{stackrel}
\usepackage[noadjust]{cite}

\usepackage[shortlabels]{enumitem}
\usepackage{comment}	
\usepackage{url}
\usepackage{hyperref}
\hypersetup{
    colorlinks=true,
    linkcolor=blue,
    citecolor=red,
    filecolor=magenta,      
    urlcolor=cyan,
    pdftitle={Locally Eventually Positive Operator Semigroups},
    bookmarks=true,
    linktocpage=true,
}

\usepackage[utf8]{inputenc}

\usepackage{amssymb}
\usepackage{amsmath}
\usepackage{verbatim, stackrel}

\usepackage{stmaryrd}
\usepackage{tikz}
\usepackage{tikz-cd}

\usepackage[shortlabels]{enumitem}
\usepackage{marginnote}

\newcommand{\bbC}{\mathbb{C}}

\newcommand{\bbN}{\mathbb{N}}

\newcommand{\bbR}{\mathbb{R}}

\newcommand{\bbZ}{\mathbb{Z}}

\newcommand{\calL}{\mathcal{L}}

\DeclareMathOperator{\one}{\mathbf{1}} 
\DeclareMathOperator{\re}{Re} 
\newcommand{\argument}{\mathord{\,\cdot\,}} 
\newcommand{\norm}[1]{\left\lVert #1 \right\rVert} 
\newcommand{\modulus}[1]{\left\lvert #1 \right\rvert} 
\newcommand{\duality}[2]{\left\langle#1\, ,\, #2\right\rangle} 
\newcommand{\dom}[1]{D\left(#1\right)} 
\DeclareMathOperator{\Ima}{Rg} 
\newcommand\restrict[1]{\raisebox{-.5ex}{$|$}_{#1}} 
\DeclareMathOperator{\abs}{abs} 
\DeclareMathOperator{\hol}{hol} 
\newcommand{\dist}[2]{\operatorname{dist}\left(#1, #2\right)} 

\newcommand{\spec}{\sigma} 
\newcommand{\Res}{\mathcal{R}} 
\newcommand{\perSpec}{\spec_{\operatorname{per}}} 
\newcommand{\perP}{P_{\operatorname{per}}} 
\newcommand{\spb}{s} 
\newcommand{\gbd}{\omega_0} 

\newcommand{\Implies}[2]{``\ref{#1} $\Rightarrow$ \ref{#2}'':}

\theoremstyle{definition}
\newtheorem{definition}{Definition}[section]
\newtheorem{remark}[definition]{Remark}
\newtheorem{remarks}[definition]{Remarks}
\newtheorem*{remark*}{Remark}
\newtheorem*{remarks*}{Remarks}

\newtheorem{assumptions}[definition]{Assumptions}

\theoremstyle{plain}
\newtheorem{proposition}[definition]{Proposition}
\newtheorem{lemma}[definition]{Lemma}
\newtheorem{theorem}[definition]{Theorem}
\newtheorem*{theorem*}{Theorem}
\newtheorem{corollary}[definition]{Corollary}

\numberwithin{equation}{section} 

\begin{document}

\title[Locally eventually positive semigroups]{Locally Eventually Positive Operator Semigroups}
\author{Sahiba Arora}
\address{Sahiba Arora, Technische Universität Dresden, Institut für Analysis, Fakultät für Mathematik, 01062 Dresden, Germany}
\email{sahiba.arora@mailbox.tu-dresden.de}
\date{\today}
\begin{abstract}
We initiate a theory of locally eventually positive operator semigroups on Banach lattices. Intuitively this means:~given a positive initial datum, the solution of the corresponding Cauchy problem becomes (and stays) positive in a part of the domain, after a sufficiently large time. A drawback of the present theory of eventually positive $C_0$-semigroups is that it is applicable only when the leading eigenvalue of the semigroup generator has a strongly positive eigenvector. We weaken this requirement and give sufficient criteria for individual and uniform local eventual positivity of the semigroup. This allows us to treat a larger class of examples by giving us more freedom on the domain when dealing with function spaces -- for instance, the square of the Laplace operator with Dirichlet boundary conditions on $L^2$ and the Dirichlet bi-Laplacian on $L^p$-spaces. Besides, we establish various spectral and convergence properties of locally eventually positive semigroups. 
\end{abstract}

\maketitle

\section{Introduction}
\label{section:introduction}
While positive $C_0$-semigroups have a rich theory with a broad variety of applications (see~\cite{Nagel1986} for a survey or the more recent~\cite{BatkaiKramarRhandi2017}), the study of semigroups that become positive only for large enough times began not too long ago. At first, eventually positive matrix semigroups were studied, for example in~\cite{Noutsos2006,NoutsosTsatsomeros2008} and~\cite{TarazagaRaydanHurman2001}. More recently, an example of such a semigroup on an infinite-dimensional space emerged from the study of the Dirichlet-to-Neumann operator on the unit circle~\cite{Daners2014}. This led to the initiation of a systematic theory of \emph{eventually positive semigroups} in~\cite{DanersGlueckKennedy2016a,DanersGlueckKennedy2016b}, where a number of applications were also presented (see also~\cite[Part~III]{Glueck2016}). The theory was subsequently expanded in~\cite{DanersGlueck2017,DanersGlueck2018b,DanersGlueck2018a}.

 A limitation of the present theory is that it is applicable only when the generator of the semigroup has a leading eigenvalue with a corresponding strongly positive eigenvector. For example, eventual positivity properties of the resolvent and the semigroup corresponding to the Dirichlet bi-Laplacian on domains which are sufficiently close to a unit ball in $\bbR^d$ (in the sense of~\cite[Theorem~5.2]{GrunauSweers1997b}) were proved in~\cite[Section~6]{DanersGlueckKennedy2016b}. However, for a general bounded domain, one need not have the existence of a strongly positive eigenvector as in~\cite[Theorem~6.3]{DanersGlueckKennedy2016b}. Nevertheless, one can still expect the eigenfunction to be positive in a part of the domain. Similarly, when dealing with non-homogeneous  boundary value problems, it may help if one restricts to \emph{nodal domains} to obtain any kind of eventual positivity.
 
In fact, this notion of \emph{local eventual positivity} is not new and examples for the same appeared several years earlier. For instance, it was shown in~\cite{GazzolaGrunau2008} that for a homogeneous biharmonic heat equation in $\bbR^d$, given a positive initial datum with compact support, the solution of the corresponding Cauchy problem is eventually positive on \emph{compact sets} and that the solution is not eventually positive on the whole domain. This result was generalized to the case when the initial data is not compactly supported but has a decaying behaviour in~\cite{FerreroGazzolaGrunau2008} and similar results for the non-homogeneous case were proved in~\cite{Berchio2008}. Recently, it was observed that fractional polyharmonic equations also exhibit such a property~\cite{FerreiraFerreira2019}. In the spirit of this, the current paper aims to study local eventual positivity for strongly continuous semigroups on general Banach lattices. To be more precise, given a $C_0$-semigroup $(e^{tA})_{t \geq 0}$ and positive operators $S$ and $T$, we are interested in the positivity of $S e^{tA} T f$ for \emph{large} times $t\geq 0$ whenever $f$ is positive. We point out that our results apply to very general classes of operators, but assume the leading spectral value is an eigenvalue -- whereas, the results in~\cite{Berchio2008,FerreiraFerreira2019,FerreroGazzolaGrunau2008,GazzolaGrunau2008} only deal with specific operators, albeit do not need the existence of eigenvalues.

\subsection*{Organization of the article}

We begin by recalling several concepts and introducing some terminology in Section~\ref{section:terminology}. Then we prove some conditions sufficient for local eventual positivity in Sections ~\ref{section:sufficient-conditions-individual} and \ref{section:sufficient-conditions-uniform}. Taking inspiration from~\cite[Section~4]{DanersGlueckKennedy2016a} and~\cite[Section~4]{DanersGlueckKennedy2016b}, we prove local eventual positivity results for the semigroup as well as the resolvent. In Section~\ref{section:applications}, we illustrate how our results from the previous sections can be applied to concrete examples. Towards the end, we deduce a few properties of local eventual positivity in Section~\ref{section:consequences-i}, and finally in Section~\ref{section:consequences-ii}, we look at various spectral and convergence implications of local eventual positivity.

Throughout the paper, we adopt several techniques and adapt many results from~\cite{DanersGlueckKennedy2016a,DanersGlueckKennedy2016b,DanersGlueck2017,DanersGlueck2018b} to our more general setting. However, even in cases where an argument closely resembles one from those references, we give full details in order to make the paper more self-contained.

\section{Terminology and preliminaries}
\label{section:terminology}
We will use the following notation throughout this article. For complex Banach spaces $E$ and $F$, we denote the space of bounded linear operators from $E$ to $F$ by $\calL(E,F)$ (and by $\calL(E)$ when $E=F$). The range of an operator $T \in \calL(E,F)$ will be denoted by $\Ima T$. As usual, the dual space of $E$ is denoted by $E'$ and if $u$ is a vector in $E$ and $\phi$ is a functional in $E'$, then we write $u \otimes \phi$ to denote the element of $\calL(E)$ given by $f \mapsto \duality{\phi}{f}u$. On function spaces, the indicator function $\one_A$ of a set $A$, denotes the function which takes value one on $A$ and zero elsewhere.

\subsection*{Banach lattices}

We expect that the reader is familiar with the theory of Banach lattices, for which we refer to the monographs~\cite{Schaefer1974,Meyer-Nieberg1991}. Let $E$ be a complex Banach lattice. Then $E$ is the complexification of a real Banach lattice (which we denote by $E_{\bbR}$); see~\cite[Appendix~C]{Glueck2016}. As is typical, the symbols $f^+, f^-$, and $|f|$ stand  respectively for the positive part, the negative part, and the modulus of $f\in E$. The \emph{positive cone} of $E$ is denoted by $E_+$ and consists of all positive elements of $E$, i.e., $E_+=\{f \in E: f \geq 0\}$. For a vector $f$ in $E$, we write $f>0$ to denote that it is both positive and non-zero. If $u$ is another positive vector in $E$, then $f$ is called \emph{strongly positive with respect to u} (written $f \gg_u 0$) if there exists $c>0$ such that $f \geq cu$. Analogously, we say $f$ is \emph{strongly negative with respect to u} and write $f\ll_u 0$ if $-f \gg_u 0$.

The \emph{principal ideal} generated by $u$ is given by
\[
	E_u:=\{f \in E: \text{ there exists } c\geq 0 \text{ such that } \modulus{f}\leq cu\}
\]
and it is endowed with the lattice norm complexification~\cite[Section~II.11]{Schaefer1974} of the \emph{gauge norm} on $\left(E_{\bbR}\right)_u$. The gauge norm is given by
\[
	\norm{f}_u := \inf \{ c\geq 0 : \modulus{f} \leq cu \}, \quad f \in E_u.
\]
Recall that the principal ideal embeds continuously into $E$. In addition, if this embedding is dense, then $u$ is said to be \emph{strongly positive} or alternatively, a \emph{quasi-interior point} of $E$ (denoted by $u \gg 0$). A vector $f$ is said to be \emph{strongly negative} ($f\ll 0$) if $-f$ is strongly positive. It is clear that, if a vector is strongly positive with respect to a quasi-interior point, then it itself is a quasi-interior point.

Now, let $F$ be another Banach lattice and $S,T\in \calL(E,F)$. Then $T$ is called \emph{positive} (denoted by $T\geq 0$) if $T$ maps the positive cone of $E$ to the positive cone of $F$ and it is called \emph{strongly positive} (denoted by $T \gg 0$) if it maps every positive non-zero element of $E$ to a quasi-interior point of $F$. Similarly, we say $T$ is \emph{strongly negative} (written as $T\ll 0$) if $-T \gg 0$. When $T$ and $S$ are real operators (defined below), then naturally, we use the shorthand $T \geq S$ and $T \gg S$ for $T-S\geq 0$ and $T-S \gg 0$ respectively. If $u \in E_+$, then we say $T$ is \emph{strongly positive with respect to $u$} (denoted by $T\gg_u 0$) if $Tf \gg_u 0$ whenever $f>0$. In a parallel manner, we say $T$ is \emph{strongly negative with respect to $u$} and write $T \ll_u 0$ if $-T \gg_u 0$.  Moreover, we say a functional $\phi \in E'$ is \emph{strictly positive} if $\duality{\phi}{f} >0$ for every $f>0$, or equivalently if $\ker \phi$ contains no positive non-zero element.

We point out that if $\phi$ is a quasi-interior point of $E'$, then it must be strictly positive but the converse is not true. Due to this, we have the following ambiguity:~Saying $\phi$ is strongly positive ($\phi \gg 0$), could either mean that $\phi$ is a quasi-interior point of $E'$ or it could mean that $\phi$ is a strongly positive as an operator from $E$ to the corresponding scalar field (which is equivalent to the strict positivity of $\phi$). To get past this, we never use the term ``strong positivity'' (or the notation $\gg$) when talking about a linear functional.

In order to prove sufficient conditions for uniform local eventual positivity in Section~\ref{section:sufficient-conditions-uniform}, it will be useful to endow $E$ with another norm:~A strictly positive functional $\phi$ on $E$ also generates a norm on $E$, given by $\norm{\argument}_{E^\phi}:= \duality{\phi}{\modulus{\argument}}$. The completion of $E$ with respect to this norm is a Banach lattice and will be denoted $E^{\phi}$. As a result, the following canonical injections are continuous and injective lattice homomorphisms:
\[
	E_u \hookrightarrow E \hookrightarrow E^{\phi}.
\]
Furthermore, the second embedding has a dense range. An important property of the space $E^{\phi}$ is that it is an \emph{AL-space}, i.e., $\norm{f+g}_{E^\phi}=\norm{f}_{E^\phi}+\norm{g}_{E^\phi}$ for all positive vectors $f$ and $g$ in $E^{\phi}$. Therefore, there exists an isometric Banach lattice isomorphism from $E^{\phi}$ to a $L^1$-space~\cite[Theorem~2.7.1]{Meyer-Nieberg1991}. 

Finally, in Section~\ref{section:consequences-ii}, we will use the concept of band projections on Banach lattices. A projection $S$ on a Banach lattice $E$ is called a \emph{band projection} if both $S$ and $I-S$ are positive operators, i.e., $0 \leq Su \leq u$ for every positive vector $u$ in $E$. Every band projection is a lattice homomorphism and each pair of band projections commutes~\cite[Theorem~II.2.9]{Schaefer1974}. 

\subsection*{Linear operators and spectral theory}

The domain of a linear operator $A$ on a Banach space $E$ is denoted by $\dom{A}$ and it is endowed with the graph norm. If $A$ is densely defined, we denote its dual by $A'$. Further, if $E$ and $F$ are Banach lattices and $A: \dom{A}\subseteq E \to F$ is linear then $A$ is said to be \emph{real} if, for every element $x+iy$ (where $x,y\in E_{\bbR}$) in $\dom{A}$, we have $x,y \in \dom{A}$ and $A$ maps the real part of $\dom{A}$ to the real part of $F$. Positive operators are natural examples of real operators.

Let $A$ be a closed linear operator on a Banach space $E$. We denote the \emph{resolvent set} and \emph{spectrum} of $A$ by $\rho(A)$ and $\spec(A)$ respectively. The \emph{spectral bound} of $A$ is defined as 
\[
	\spb(A):= \sup \{\re \lambda : \lambda \in \spec(A)\} \in [-\infty,\infty]
\]
and the \emph{peripheral spectrum} of $A$  is given by
\[
	\perSpec(A):= \spec(A) \cap (\spb(A) + i\bbR);
\]
whenever $\spb(A)$ is a real number.

We say that $\spb(A)$ is a \emph{dominant spectral value} if $\perSpec(A) = \{\spb(A)\}$. The resolvent of $A$ at ${\lambda \in \rho(A)}$ is the bounded linear operator ${\Res(\lambda,A):E \to \dom{A}\subseteq E}$ given by $\Res(\lambda, A)=(\lambda-A)^{-1}$. One can immediately see that the resolvent $\Res(\lambda,A)$ is real if both $A$ and $\lambda$ are real. We will make extensive use of the spectral theory of closed linear operators -- especially the Laurent series expansion of the resolvent $\Res(\argument,A)$ about its poles -- and the spectral projection of the resolvent $\Res(\argument,A)$ associated to a bounded isolated set;  see the classical references~\cite[Chapter~VIII]{Yosida1980},~\cite[Section~III.6]{Kato1980},~\cite[Chapter~V]{TaylorLay1986}, and~\cite[Section~IV.1]{EngelNagel2000} or alternatively~\cite[Appendix~A]{Glueck2016}. Let $\lambda_0$ be an isolated spectral value of $A$ and a first-order (\emph{simple}) pole of the resolvent $\Res(\argument,A)$. If $P$ denotes the associated spectral projection, then an important fact that we make use of throughout, without quoting is that $\Ima P = \ker(\lambda_0 -A)$ and $\Ima P'=\ker(\lambda_0-A')$; see~\cite[Theorem~VIII.8.3]{Yosida1980}.

\subsection*{Operator semigroups and eventual positivity} 

Finally, we assume the reader is also acquainted with the theory of $C_0$-semigroups. A standard reference for this is~\cite{EngelNagel2000} or~\cite[Part~A]{Nagel1986}. For a $C_0$-semigroup generated by an operator $A$ on a Banach space $E$, we use the notation $(e^{tA})_{t \geq 0}$ and denote its \emph{growth bound} by $\gbd(A)$. Recall that the growth bound of the semigroup is always larger than or equal to the spectral bound of the generator~\cite[Corollary~II.1.13]{EngelNagel2000}. The dual operator $A'$ is the weak$^*$-generator of the dual semigroup $\left((e^{tA})'\right)_{t \geq 0}$ on $E'$, which is continuous in the weak$^*$-topology (see~\cite[Section~II.2.5]{EngelNagel2000} for a brief overview or the monograph~\cite{Neerven1992} for extensive literature on dual semigroups).

 If $E$ is a Banach lattice, then $(e^{tA})_{t \geq 0}$ is said to be a \emph{real semigroup} if each operator $e^{tA}$ is real. Clearly, a semigroup is real if and only if its generator is real. Throughout the paper, we deal with real semigroups. From an application perspective, this condition is mild:~for instance, every $C_0$-semigroup generated by a differential operator with real coefficients is real. 
 
Lastly, recall that a real semigroup is called \emph{individually eventually strongly positive with respect to $u \in E_+$} if, for each positive non-zero vector $f$ in $E$, there exists a time $t_0\geq 0$ such that $e^{tA}f \gg_u 0$ for all $t \geq t_0$. If the time $t_0$ can be chosen independently of $f$, then we say that the semigroup is  \emph{uniformly eventually strongly positive with respect to $u$}. In~\cite[Examples~5.7 and 5.8]{DanersGlueckKennedy2016b}, it was shown that individual and uniform eventual positivity are not equivalent on infinite-dimensional spaces. Various notions of eventual positivity of both the semigroup and the resolvent are defined and characterized in~\cite{DanersGlueckKennedy2016a} and~\cite{DanersGlueckKennedy2016b}. Moreover, a sufficient condition for uniform eventual positivity of the semigroup is proved in~\cite[Theorem~3.1]{DanersGlueck2018b}.

\subsection*{Main definitions}

We end this section by introducing the central notions of this article. But first, we state some assumptions, which will remain our standard assumptions all through Sections~\ref{section:sufficient-conditions-individual} and \ref{section:sufficient-conditions-uniform}. Therefore it is helpful to collect them here:
\begin{assumptions}
	\label{ass:standing-assumptions}
	Let $E, F$, and $G$ be complex Banach lattices. Additionally:
	\begin{enumerate}[\upshape (a)]
		\item Let $A$ be a closed, densely defined, and real operator on $F$.
		
		\item Let $S \in \calL(F,G)$ and $T\in \calL(E,F)$ be positive operators.
		
		\item Let $u \in F$ and $\phi \in F'$ be such that $Su$ is a positive vector in $G$ and $T'\phi$ is a strictly positive functional on $E$.
	\end{enumerate}
\end{assumptions}
The following diagram illustrates the situation in Assumptions~\ref{ass:standing-assumptions}(b) and (c):

\begin{center}
			\begin{tikzcd}[contains/.style = {draw=none,"\in" description,sloped}]
						     & & u\arrow[contains]{d} \arrow[mapsto]{rr} & & Su \arrow[contains]{d} 	\\
				E \arrow{rr}{T} & & F \arrow{rr}{S} & & G \\
				E'  & & F' \arrow{ll}{T'} & & G' \arrow{ll}{S'}\\
				\arrow[contains]{u}T'\phi & & \arrow[mapsto]{ll} \phi\arrow[contains]{u} & &
			\end{tikzcd}
\end{center}

We are now ready to define local eventual positivity of the resolvent (cf.~\cite[Definition~4.1]{DanersGlueckKennedy2016b}). We again point out that in view of~\cite[Examples~5.7 and 5.8]{DanersGlueckKennedy2016b}, we must distinguish between individual and uniform local eventual positivity.

\begin{definition}
	Let Assumptions~\ref{ass:standing-assumptions} be fulfilled and $\lambda_0$ be either $-\infty$ or a real spectral value of $A$.
	\begin{enumerate}[\upshape (a)]
	\item We say $S \Res(\argument,A) T$ is \emph{individually eventually strongly positive with respect to $Su$ at $\lambda_0$} if for every $0<f \in E$, there exists $\lambda_1>\lambda_0$ with the following properties:~$(\lambda_0,\lambda_1]\subseteq \rho(A)$ and $S\Res(\lambda, A)Tf\gg_{Su} 0$ for all $\lambda \in (\lambda_0,\lambda_1]$.
	
	\item We say $S \Res(\argument,A) T$ is \emph{uniformly eventually strongly positive with respect to $Su$ at $\lambda_0$} if  there exists $\lambda_1>\lambda_0$ such that ${(\lambda_0,\lambda_1]\subseteq \rho(A)}$ and the operator ${S\Res(\lambda, A)T\gg_{Su} 0}$ for all $\lambda \in (\lambda_0,\lambda_1]$.
	\end{enumerate}
\end{definition}

Let $\Omega \subseteq \bbR^d$ be a bounded domain with sufficiently smooth boundary and $A_p$ be the \emph{Dirichlet bi-Laplacian} on $L^p(\Omega,\bbC)$, for some $1<p<\infty$. In Section~\ref{section:applications}, we will show that (under technical assumptions) for large $p$, the family $S \Res(\argument, A_p)T$ is individually eventually strongly positive with respect to ${S(\one_\Omega)=\one_K}$ at $\spb(A_p)$; here ${S:L^p(\Omega,\bbC)\to L^p(\Omega,\bbC)}$ is defined as $Sf = \one_K f$ for all $f\in L^p(\Omega,\bbC)$ and for a suitable compact set $K$ and $T:\{f\restrict{K}:f\in L^p(\Omega,\bbC)\}\to L^p(\Omega,\bbC)$ is the extension by zero operator. In particular, if we start with an arbitrary positive function $f$ supported on $K$, then $R(\lambda,A_p)f(y)>0$ for every $y\in K$ and for all $\lambda$ on some right neighbourhood of $\spb(A_p)$. That is, we get eventual positivity of the resolvent (locally) on a compact set. In this way, the operators $S$ and $T$ in the above definition are related to the intuitive notion of \emph{local eventual positivity}. 

While eventual positivity of the resolvent is concerned with the behaviour on a small interval to the right of a spectral value, analogous to~\cite[Definition~4.2]{DanersGlueckKennedy2016b}, we also define eventual negativity of the resolvent in our setting which has to do with an interval on the left of a spectral value. This will help us state a local uniform anti-maximum principle in Theorem~\ref{thm:local-anti-maximum}. 

\begin{definition}
	Let Assumptions~\ref{ass:standing-assumptions} be fulfilled and $\lambda_0$ be either $\infty$ or a real spectral value of $A$.
	\begin{enumerate}[\upshape (a)]
	\item We say $S \Res(\argument,A) T$ is \emph{individually eventually strongly negative with respect to $Su$ at $\lambda_0$} if for every $0<f \in E$, there exists $\lambda_1<\lambda_0$ with the following properties:~${[\lambda_1,\lambda_0)\subseteq \rho(A)}$ and $S\Res(\lambda, A)Tf\ll_{Su} 0$ for all $\lambda \in [\lambda_1,\lambda_0)$.
	
	\item We say $S \Res(\argument,A) T$ is \emph{uniformly eventually strongly negative with respect to $Su$ at $\lambda_0$} if  there exists $\lambda_1<\lambda_0$ such that ${[\lambda_1,\lambda_0)\subseteq \rho(A)}$ and the operator ${S\Res(\lambda, A)T\ll_{Su} 0}$ for all $\lambda \in [\lambda_1,\lambda_0)$.
	\end{enumerate}
\end{definition}

Lastly, we define local eventual positivity of the semigroup. 

\begin{definition}
	Let Assumptions~\ref{ass:standing-assumptions} be fulfilled and in addition, let $A$ generate a $C_0$-semigroup on $F$.
	\begin{enumerate}[\upshape (a)]
	\item The family $(S e^{tA} T)_{t\geq 0}$ is called \emph{individually eventually strongly positive with respect to $Su$} if for every $0<f \in E$, there exists $t_0\geq 0$ such that $Se^{tA}Tf\gg_{Su} 0$ for all $t\geq t_0$.
	\item The family $(S e^{tA} T)_{t\geq 0}$ is called \emph{uniformly eventually strongly positive with respect to $Su$} if there exists $t_0\geq 0$ such that $Se^{tA}T\gg_{Su} 0$ for all $t\geq t_0$.
	\end{enumerate}
\end{definition}

We note that if, in particular, $E=F=G$ and $S$ and $T$ are the identity operator, then the above definitions reduce to the corresponding eventual positivity concepts presented in~\cite{DanersGlueckKennedy2016b}.

\section{Sufficient conditions for individual local eventual positivity}
\label{section:sufficient-conditions-individual}

In this section, we prove criteria sufficient for individual local eventual positivity of the semigroup and the resolvent. In order to develop characterizations for individual eventual positivity (under appropriate technical assumptions), the approach followed in both~\cite{DanersGlueckKennedy2016a} and~\cite{DanersGlueckKennedy2016b}, was to first obtain equivalent conditions for the spectral projection associated to a pole to be strongly positive. In the same way, we start by stating the conditions that allow the spectral projection associated to a \emph{simple} pole to be locally strongly positive.

\begin{proposition}
	\label{prop:sufficient-projection}
	Let Assumptions~\ref{ass:standing-assumptions} be fulfilled. Suppose $\lambda_0 \in \bbR$ is a spectral value of $A$ and a simple pole of the resolvent $\Res(\argument,A)$ such that the eigenspace ${\ker(\lambda_0-A)}$ is one-dimensional. Let $x$ be a non-zero vector in $\ker(\lambda_0-A)$ and $\psi$ be a non-zero functional in the dual eigenspace $\ker(\lambda_0-A')$. Denote by $P$ the spectral projection associated to $\lambda_0$.
	
	 If $Sx \gg_{Su} 0$ and $T'\psi \gg_{T'\phi} 0$, then there exists a constant $c>0$ such that ${SPT \geq c\, Su \otimes T'\phi}$. In particular, $SPT$ is strongly positive with respect to $Su$.
\end{proposition}

\begin{proof}
	As the spectral value $\lambda_0$ is a simple pole of the resolvent $\Res(\argument,A)$, we have $\Ima P=\ker(\lambda_0-A)$ and $\Ima P'=\ker(\lambda_0-A')$. Rescaling $x$ such that $\duality{\psi}{x}=1$ gives $P = x \otimes \psi$. Finally, since $Sx \gg_{Su} 0$ and $T'\psi \gg_{T'\phi} 0$, there exist $c_1,c_2>0$ such that $Sx \geq c_1 Su$ and $T'\psi \geq c_2 T'\phi$. This yields
	\[
		SPT=Sx\otimes T'\psi\geq (c_1 c_2)\, Su \otimes T'\phi.
	\]
	Taking $c=c_1c_2$ and keeping in mind the strict positivity of $T'\phi \in E'$ completes the proof.
\end{proof}

If $v$ is a quasi-interior point of $F$, then an important condition which was imposed in the characterization of individual eventual positivity (and negativity) of the resolvent $\Res(\argument ,A)$ in~\cite[Theorem~4.4]{DanersGlueckKennedy2016b} is the \emph{domination condition}
\[
	\dom{A} \subseteq F_v.
\]
This condition was  analysed in~\cite{DanersGlueck2017}, where it was elucidated that certain implications of~\cite[Theorem~4.4]{DanersGlueckKennedy2016b} remain true even without the domination condition. However, the domination condition cannot be completely dropped~\cite[Example~5.4]{DanersGlueckKennedy2016b}. Nevertheless, from an application point of view, this condition is not too restrictive. For instance, it always holds on the space of continuous (complex-valued) functions on a compact Hausdorff space, equipped with the supremum norm. More concrete examples where such a condition is satisfied can be found in~\cite[Example~4.5 and Section~6]{DanersGlueckKennedy2016b}. In the following theorem, we impose a \emph{local} version of the domination condition to obtain local eventual positivity of the resolvent at a \emph{simple} pole. As a matter of fact,  we show that as in~\cite[Theorem~4.4]{DanersGlueckKennedy2016b}, our conditions are also sufficient to guarantee local eventual negativity of the resolvent.

\begin{theorem}
	\label{thm:sufficient-individual-resolvent}
	Let Assumptions~\ref{ass:standing-assumptions} hold. Suppose $\lambda_0 \in \bbR$ is a spectral value of $A$ and a simple pole of the resolvent $\Res(\argument,A)$ such that the eigenspace $\ker(\lambda_0-A)$ is one-dimensional. Let $x$ be a non-zero vector in $\ker(\lambda_0-A)$ and $\psi$ be a non-zero functional in the dual eigenspace $\ker(\lambda_0-A')$. Moreover, assume $Sx \gg_{Su} 0$ and $T'\psi \gg_{T'\phi} 0$. 
	
	If $S\dom{A}\subseteq G_{Su}$, then for each $0<f\in E$, there exists an open interval $U$ containing $\lambda_0$ and $c>0$ such that $U\setminus \{\lambda_0\}\subseteq \rho(A)$ and $(\lambda-\lambda_0)S\Res(\lambda,A)Tf\geq c S u$ for all $\lambda \in U\setminus \{\lambda_0\}$. In particular, $S\Res(\argument,A)T$ is individually eventually strongly positive and negative  with respect to $Su$ at $\lambda_0$.
\end{theorem}

We point out that the constant $c>0$ in the above theorem is dependent on $f\in E$ but not on $\lambda \in U$.

\begin{proof}[Proof of Theorem~\ref{thm:sufficient-individual-resolvent}]
	Replacing $A$ by $A-\lambda_0$ if necessary, we may assume without loss of generality that $\lambda_0=0$. Let $P$ denote the spectral projection corresponding to $0$. Since $0$ is a simple pole of the resolvent $\Res(\argument,A)$, we have ${\Ima P=\ker A}$ and ${\lim_{\lambda\downarrow 0} \lambda \Res(\lambda,A) = P}$ in $\calL(F)$. Fix $\mu \in \rho(A)$, then ${\Res(\mu,A)}$ is a bounded operator from $F$ to $\dom{A}$ and $\mu\Res(\mu,A)P=P$. By the resolvent equation, we get
	\begin{align*}
		\lambda  \Res(\lambda,A) &=\lambda  \Res(\mu,A) + (\mu-\lambda)\Res(\mu,A)\lambda \Res(\lambda,A)\\
				& \to \mu  \Res(\mu,A)P=P
	\end{align*}
	in $\calL(F,\dom{A})$ as $\lambda\downarrow 0$. Therefore $\lambda S\Res(\lambda,A)T\to SPT$ in $\calL(E,S\dom{A})$ as $\lambda\downarrow 0$. Because $S\dom{A}\subseteq G_{Su}$, we have $S\in \calL(\dom{A},G_{Su})$ due to the closed graph theorem. This implies that the above convergence holds in $\calL(E,G_{Su})$ as well. 
	
	Next, let $0<f \in E$. Then by Proposition~\ref{prop:sufficient-projection}, there exists $c>0$ such that $SPTf \geq cSu$. Since the operators $S,T$, and $A$ are real, so is the vector $\lambda S\Res(\lambda,A)Tf$ for every $\lambda \in \rho(A)\cap \bbR$. Hence using the above convergence, there exists $\delta>0$ such that
	\[
		\norm{\lambda S\Res(\lambda,A)Tf-SPTf}_{Su} <\frac{c}{2} \text{ for all } 0<\lambda<\delta.
	\]
	Thus for $\lambda \in (0,\delta)$, we obtain
	\[
		\modulus{\lambda S\Res(\lambda,A)Tf-SPTf} \leq \norm{\lambda S\Res(\lambda,A)Tf-SPTf}_{Su} Su < \frac{c}{2} Su 
	\]
	and so
	\[
		\lambda S\Res(\lambda,A)Tf>SPTf- \frac{c}{2} Su \geq c Su -\frac{c}{2} Su =\frac{c}{2} Su.
	\]
	
	Finally, note that that $0$ is also an eigenvalue of $-A$ and $P$ is the corresponding spectral projection. Therefore, applying the above arguments to the operator $-A$ ensures the existence of $\delta'>0$ such that $\lambda S\Res(\lambda,-A)Tf\geq \frac{c}{2} S u$ for all $\lambda \in (0,\delta')$. Taking $U=(-\delta',\delta)$ and replacing $c/2$ by $c$ makes it possible to conclude that $\lambda S\Res(\lambda,A)Tf\geq c S u$ for all $\lambda \in U\setminus \{0\}$. In particular, $S\Res(\argument,A)T$ is individually eventually strongly positive and negative  with respect to $Su$ at $0$.
\end{proof}

It was proved (under certain technical assumptions) in~\cite[Theorem~5.2]{DanersGlueckKennedy2016b} that if $A$ generates a $C_0$-semigroup on a complex Banach lattice $F$, then  individual eventual strong positivity of $(e^{tA})_{t \geq 0}$  with respect to a quasi-interior point $v\in F$ is equivalent to the strong convergence of the (rescaled) semigroup operators to $P\gg_v 0$ as $t\to \infty$ (here $P$ is the spectral projection associated to the spectral bound $\spb(A)$). In comparison to the case of resolvent -- instead of a \emph{domination} condition (i.e., $D(A) \subseteq F_v$) -- this equivalence was proved under a so-called \emph{smoothing} assumption 
\[
	\exists\ t_0 > 0 \text{ such that }  e^{t_0 A} F \subseteq  F_{v}.
\]
As before, in view of~\cite[Example~5.4]{DanersGlueckKennedy2016b} this condition cannot entirely be removed (although a certain part of the theorem remains true even if drop the condition~\cite[Theorem~5.1]{DanersGlueck2017}). Taking cue from this, we a priori assume a version of the smoothing condition in our setting and that the (rescaled) semigroup operators converge in the strong operator topology as $t \to \infty$. This helps us adapt the  proof of~\cite[Theorem~5.2]{DanersGlueckKennedy2016b} to get a sufficient criterion for individual local eventual positivity of the semigroup. However, unlike~\cite[Theorem~5.2]{DanersGlueckKennedy2016b}, we do not impose any restrictions on the peripheral spectrum of $A$.

As was true for the domination condition, the smoothing condition is trivially true on the space of continuous (complex-valued) functions on a compact Hausdorff space, equipped with the supremum norm. We will see more examples of when this condition holds in Section~\ref{section:applications}.  Even more situations (albeit with $S$ and $T$ as the identity operators) where such a condition remains true are discussed in~\cite[Remark~3.3]{DanersGlueck2018b} and~\cite[Remark~9.3.4]{Glueck2016}, while concrete examples can be found in~\cite[Section~6]{DanersGlueckKennedy2016b} and~\cite[Section~4]{DanersGlueck2018b}.

\begin{theorem}
	\label{thm:sufficient-individual-semigroup-resolvent}
	Under Assumptions~\ref{ass:standing-assumptions}, let $A$ generate a $C_0$-semigroup $(e^{tA})_{t \geq 0}$ on $F$. In addition, assume $\spb(A)>-\infty$ and that the rescaled semigroup $\left(e^{t(A-\spb(A))}\right)_{t \geq 0}$ converges in the strong operator topology as $t\to \infty$. Suppose that the spectral bound $\spb(A)$ is an eigenvalue of $A$ such that the corresponding eigenspace ${\ker(\spb(A)-A)}$ is one-dimensional. Let $x$ be a non-zero vector in $\ker(\spb(A)-A)$ and $\psi$ be a non-zero functional in the dual eigenspace $\ker(\spb(A)-A')$. Further, assume $Sx \gg_{Su} 0$ and $T'\psi \gg_{T'\phi} 0$. Then the following assertions are true.
	\begin{enumerate}[ref=(\roman*)]
		\item\label{thm:sufficient-individual:semigroup} If there exists $t_0 >0$ such that $S e^{t_0A} F \subseteq G_{Su}$, then for each ${0<f\in E}$, there exists $t_1\geq 0$ and $c>0$ such that $Se^{tA}Tf\geq c Su$ for all ${t\geq t_1}$. In particular, the family $(S e^{tA} T)_{t\geq 0}$ is individually eventually strongly positive with respect to $Su$.
		
		\item\label{thm:sufficient-individual:resolvent} If $S\dom{A} \subseteq G_{Su}$, then $S\Res(\argument,A)T$ is individually eventually strongly positive with respect to $Su$ at $\spb(A)$.
	\end{enumerate}
\end{theorem}

We emphasize here that the constant $c>0$ in \ref{thm:sufficient-individual:semigroup} is not dependent on the time $t\geq t_1$.

\begin{proof}[Proof of Theorem~\ref{thm:sufficient-individual-semigroup-resolvent}]
	Without loss of generality, we may assume $\spb(A)=0$. Let
	\[
		Pf:= \lim_{t \to \infty} e^{tA}f \text{ for every } f \in E.
	\]
	Then by~\cite[Corollary~4.3.2 and Proposition~4.3.4(a)]{ArendtBattyHieberNeubrander2011}, $P$ is the projection onto $\ker A$ along $\overline{\Ima A}$ and $\ker A$ separates $\ker A'$. Since $P$ is a projection, we have ${\dim P=\dim P'}$ (see~\cite[Lemma~1.2.6]{Glueck2016} or~\cite[Section~III.6.6]{Kato1980}). Thus $\Ima P =\ker A$ and ${\Ima P'=\ker A'}$. Now rescaling $x$ such that $\duality{\psi}{x}=1$, we have $P = x \otimes \psi$.
	
	\ref{thm:sufficient-individual:semigroup} Suppose there exists $t_0>0$ such that $Se^{t_0A}F\subseteq G_{Su}$, then by the closed graph theorem $Se^{t_0A} \in \calL(F,G_{Su})$. Therefore
	\[
		Se^{tA}T=Se^{t_0A}e^{(t-t_0)A} T \to S e^{t_0A} PT
	\]
	 as $t \to \infty$ in $\calL(E,G_{Su})$. Because $\Ima P =\ker A$, we have $e^{t_0A}P=P$ and so ${Se^{tA}T \to SPT}$ as $t \to \infty$ in $\calL(E,G_{Su})$. 
	
	Let $0<f \in E$. Since the operators $S, T$, and $e^{tA}$ are real, so is the vector $Se^{tA}Tf$ for every $t \geq 0$. By assumption, there exist $c_1,c_2>0$ such that $Sx \geq c_1 Su$ and $T'\psi \geq c_2 T'\phi$. Then $c:=c_1c_2 \duality{T'\phi}{f}>0$, because  $T'\phi$ is strictly positive. Hence $SPTf = (Sx\otimes T'\psi) f\geq c Su$. Moreover, using the above convergence we can find $t_1>0$ such that
	\[
		\norm{Se^{tA}Tf-SPTf}_{Su} < \frac{c}{2} \text{ for all } t \geq t_1.
	\]
	Therefore for $t \geq t_1$,
	\[
		\modulus{Se^{tA}Tf-SPTf}\leq \norm{Se^{tA}Tf-SPTf}_{Su} Su <\frac{c}{2} Su
	\]
	and so
	\[
		Se^{tA}Tf > SPTf-\frac{c}{2} Su \geq c Su -\frac{c}{2} Su = \frac{c}{2} Su.
	\]
	Whence (replacing $c/2$ with $c$) the assertion follows.
	
	\ref{thm:sufficient-individual:resolvent}  The strong convergence of the semigroup operators implies the strong convergence of $\lambda \Res(\lambda,A) \to P$ as $\lambda \downarrow 0$ in by the Laplace transform representation of the resolvent and the Abelian theorem~\cite[Proposition~4.1.2]{ArendtBattyHieberNeubrander2011}. Therefore if $S\dom{A}\subseteq G_{Su}$, then we may conclude that $S\Res(\argument,A)T$ is individually eventually strongly positive with respect to $Su$ at $0$ exactly as in the proof of Theorem~\ref{thm:sufficient-individual-resolvent}.
\end{proof}

\begin{remarks}
	\label{rem:sufficient-individual}
	
	(a) In the above theorem, if we had assumed that $\spb(A)$ is a pole of the resolvent, then the proof becomes slightly shorter:~Using spectral theory, $P$ is the spectral projection associated to $\spb(A)$. Moreover, convergence of semigroup operators implies that the semigroup is bounded and hence one can show that $\spb(A)$ is a simple pole. Then Proposition~\ref{prop:sufficient-projection} yields $SPT \gg_{Su} 0$. The proof can now be repeated to conclude the assertions. 
	
	(b) The strong convergence of $(e^{tA})_{t \geq 0}$ (assuming $\spb(A)=0$) in the above theorem is equivalent to the following condition \cite[Lemma~V.4.4]{EngelNagel2000}:~the semigroup $(e^{tA})_{t \geq 0}$ is bounded and mean ergodic with mean ergodic projection $P$ and $e^{tA} \to 0$ as $t\to \infty$ strongly on $\ker P$.
	
	(c) If $(e^{tA})_{t \geq 0}$ is eventually differentiable, then the assumption -- there exists ${t_0 >0}$ such that $S e^{t_0 A} F \subseteq G_{Su}$ -- in Theorem~\ref{thm:sufficient-individual-semigroup-resolvent}\ref{thm:sufficient-individual:semigroup} is implied by the following condition:~there exists $n\in \bbN$ such that $S\dom{A^n} \subseteq G_{Su}$ (cf.~\cite[Corollary~5.3]{DanersGlueckKennedy2016b}). Indeed, if the semigroup is eventually differentiable, then there exists $t_0 \geq 0$ such that ${e^{t_0 A} F\subseteq \dom{A}}$, which gives $Se^{nt_0 A} F \subseteq S \dom{A^n} \subseteq G_{Su}$.
\end{remarks}

It is instructive to observe that local eventual positivity of the resolvent requires the stronger assumption $S \dom{A} \subseteq G_{Su}$ as compared to (eventually differentiable) semigroups which only ask $S\dom{A^n}\subseteq G_{Su}$ for some $n\in \bbN$. Due to this, we will only be able to show local eventual positivity results about the semigroup generated by the square of the Dirichlet Laplacian in Theorem~\ref{thm:square-dirichlet-laplacian} (see also~\cite[Proposition~6.5 and 6.6]{DanersGlueckKennedy2016b}). When proving local (anti-)maximum principles in Section~\ref{section:sufficient-conditions-uniform}, it will be further highlighted how the (local) eventual behaviour of the resolvent is, in general, more demanding.

We now show that in Theorem~\ref{thm:sufficient-individual-semigroup-resolvent}, we can weaken our requirement of -- strong convergence of the semigroup operators in $F$ -- to -- strong convergence of the semigroup operators  in a small part. The advantage is that, in applications, with the help of Tauberian theorems (for example the ABLV theorem in~\cite[Theorem~2.4]{ArendtBatty88} or~\cite[Theorem~V.2.21]{EngelNagel2000}), the strong convergence in a part of the domain is much easier to show. This will be illustrated in Section~\ref{section:applications}.

\begin{corollary}
	\label{cor:sufficient-individual-semigroup-resolvent}
	Let Assumptions~\ref{ass:standing-assumptions} hold and let $A$ generate a $C_0$-semigroup $(e^{tA})_{t \geq 0}$. Furthermore, assume that the spectral bound $\spb(A)$ is an eigenvalue of $A$ and that the corresponding eigenspace $\ker(\spb(A)-A)$ is one-dimensional. Let $x$ be a non-zero vector in $\ker(\spb(A)-A)$ and $\psi$ be a non-zero functional in the dual eigenspace $\ker(\spb(A)-A')$.
	
	In addition, suppose the restriction semigroup $\left(e^{t(A-\spb(A))}\restrict{\ker \psi}\right)_{t \geq 0}$ converges to $0$ in the strong operator topology as $t \to \infty$, i.e.,
	\[
		\lim_{t \to \infty }e^{t(A-\spb(A))} f =0 \text{ for all } f \in \ker \psi.
	\] 
	Then the rescaled semigroup $\left(e^{t(A-\spb(A))}\right)_{t \geq 0}$ converges strongly as $t\to\infty$. In particular, if $Sx \gg_{Su} 0$ and $T'\psi \gg_{T'\phi} 0$ then both the assertions \ref{thm:sufficient-individual:semigroup} and \ref{thm:sufficient-individual:semigroup} of Theorem~\ref{thm:sufficient-individual-semigroup-resolvent} hold.
\end{corollary}

\begin{proof}
	We assume without any loss of generality that $\spb(A)=0$. According to Theorem~\ref{thm:sufficient-individual-semigroup-resolvent}, if we are able to show that the semigroup $(e^{tA})_{t \geq 0}$ converges strongly as $t \to \infty$, then the latter assertion holds as well. To prove convergence, first note that since $x$ is a vector in $\ker A$, therefore $e^{tA}x=x$. This, together with the assumptions that $x$ is non-zero and $e^{tA}$ converges strongly to zero on $\ker \psi$ implies that $x\notin \ker \psi$. Also, the co-dimension of $\ker \psi$ is equal to one, so it follows that 
	 \[
	 	F = \ker A \oplus \ker \psi.
	 \]
	 Now, the semigroup acts as the identity operator on $\ker A$ and converges strongly to zero on $\ker \psi$. As a result, $(e^{tA})_{t \geq 0}$ converges in the strong operator topology as $t\to\infty$.
\end{proof}

\section{Sufficient conditions for uniform local eventual positivity}
\label{section:sufficient-conditions-uniform}

Whereas obtaining sufficient conditions for individual local eventual positivity was the content of Section~\ref{section:sufficient-conditions-individual}, we now shift our focus to uniform local eventual positivity. Daners and Gl\"uck proved a criterion sufficient for uniform eventual positivity of the semigroup in~\cite{DanersGlueck2018b}. They showed that, in addition to the assumptions required for individual eventual strong positivity, a similar smoothing condition on the dual semigroup along with strict positivity of an eigenvector of $A'$ is the right setting to get uniform eventual strong positivity. We will follow the same strategy in Theorem~\ref{thm:sufficient-uniform-semigroup} and show that we can obtain uniform local eventual positivity of semigroups that are norm continuous at infinity by adjusting these assumptions in our setting, as well.

On the other hand, the story with resolvents is quite different. First of all, it was shown in~\cite[Propositions~4.2 and 4.3]{DanersGlueckKennedy2016a}, that on the spaces of continuous functions, if we a priori know that the resolvent is positive at one point, then one can obtain uniform eventual positivity of the resolvent. In fact, similar reasoning shows that this result is also true on general Banach lattices \cite[Proposition~3.2]{DanersGlueck2018a}. Unfortunately, the arguments cannot be adapted to the local case. Next, the idea which works for semigroups can also not be used as a blueprint in case of resolvents (that is, assuming a domination condition on both the operator as well as its dual). Indeed, due to a uniform anti-maximum principle given in~\cite[Theorem~1]{GrunauSweers2001}, with ${n=m=k=1}$, it can be seen that one does not obtain uniform eventual strong negativity of the resolvent even though the additional conditions on the dual generator are satisfied. This suggests that one requires even more additional conditions  in the case of resolvents. To this end,  we consider the approach followed by Tak\'a\v{c}, who proved a uniform anti-maximum principle in~\cite[Theorem~5.2]{Takac1996} by imposing a bound on the resolvent at a point. We show that a similar bound in our situation is the appropriate additional assumption to guarantee uniform local eventual negativity of the resolvent in Theorem~\ref{thm:local-anti-maximum} (which we call the \emph{local uniform anti-maximum principle}). As a consequence, we also obtain a \emph{local uniform maximum principle} in Corollary~\ref{cor:local-maximum}.

\begin{theorem}[Local uniform anti-maximum principle]
	\label{thm:local-anti-maximum}
	In addition to Assumptions~\ref{ass:standing-assumptions}, suppose the following hold:
	\begin{enumerate}[\upshape (a)]
		\item The operator $A$ has a real spectral value $\lambda_0$ which is a simple pole of the resolvent $\Res(\argument, A)$.
		
		\item The eigenspace $\ker(\lambda_0-A)$ is one-dimensional. Let $x \in \ker(\lambda_0-A)$ and $\psi \in \ker(\lambda_0-A')$ be non-zero.
		
		\item The domains of $A$ and $A'$ satisfy:~${S\dom{A} \subseteq G_{Su}}$ and ${T'\dom{A'}\subseteq (E')_{T'\phi}}$. 
		
		\item The vector $Sx$ is strongly positive with respect to $Su$ and the functional $T'\psi$ is strongly positive with respect $T'\phi$.
	\end{enumerate}
	
	Denote by $P$ the spectral projection corresponding to $\lambda_0$. If there exist $\lambda_1$ in the resolvent set of $A$ with ${\lambda_1<\lambda_0}$ and $c>0$ such that ${S \Res(\lambda_1,A)T \leq c SPT}$, then there exists $\lambda_2<\lambda_0$ satisfying the following:~$[\lambda_2,\lambda_0)\subseteq \rho(A)$ and for each $0<f\in E$, there exists $c'>0$ such that $(\lambda-\lambda_0) S\Res(\lambda,A)Tf \geq c'Su$ for all $\lambda \in[\lambda_2,\lambda_0)$. In particular, $S\Res(\argument,A)T$ is uniformly eventually strongly negative  with respect to $Su$ at $\lambda_0$.
\end{theorem}

The essence of Theorem~\ref{thm:local-anti-maximum} is that if we have a \emph{positive} upper bound for $S \Res(\argument,A) T$ at a point on the left of $\lambda_0$, then we can obtain a \emph{negative} upper bound for $S \Res(\argument,A) T$ in an entire neighbourhood on the left of $\lambda_0$. In addition, the same constant $c'>0$ works for each $\lambda \in[\lambda_2,\lambda_0)$.

The proof of Theorem~\ref{thm:local-anti-maximum} requires the following simple lemma which was proved in~\cite[Proposition~2.1]{DanersGlueck2018b} for the particular case $E=F$.

\begin{lemma}
	\label{lem:operator-extension}
	Let $T: E \to F$ be a bounded linear operator between complex Banach lattices $E$ and $F$. If for some strictly positive functional $\phi \in E'$, the operator $T'$ maps $F'$ into $(E')_{\phi}$, then $T$ can be extended to a bounded linear operator from $E^{\phi}$ to $F$ (denoted again by $T$).
\end{lemma}

\begin{proof}
	Note that $T'\in \calL\left(F',(E')_{\phi}\right)$ due to the closed graph theorem. For ${\psi \in F'}$ with $\norm{\psi}_{F'}\leq 1$, we have ${\modulus{T'\psi} \leq \norm{T'\psi}_{\phi} \phi \leq \norm{T'}_{F'\to (E')_{\phi}} \phi}$. The conclusion now follows because
	\begin{align*}
		\norm{Tf}&=\sup_{\norm{\psi}\leq 1} \modulus{\duality{\psi}{Tf}}
			     =\sup_{\norm{\psi}\leq 1} \modulus{\duality{T'\psi}{f}}\\
			     &\leq \norm{T'}_{F'\to (E')_{\phi}} \duality{\phi}{\modulus{f}}\\
			     &=\norm{T'}_{F'\to (E')_{\phi}}\norm{f}_{E^{\phi}}.
	\end{align*}
	 for all $f \in E$.
\end{proof}

\begin{proof}[Proof of Theorem~\ref{thm:local-anti-maximum}]
	Without loss of generality, we only consider the case $\lambda_0=0$. The assumption ${S\dom{A} \subseteq G_{Su}}$ and the closed graph theorem together imply that ${S \in \calL(\dom{A},G_{Su})}$. Moreover, $\Res(\mu,A)T$ extends as a bounded linear operator from $E^{T'\phi}$ to $F$ for every $\mu \in \rho(A)$ (Lemma~\ref{lem:operator-extension}). Thus, another application of the closed graph theorem yields ${S\Res(\lambda,A)\Res(\mu,A)T \in \calL\left(E^{T'\phi},G_{Su}\right)}$ for every $\lambda,\mu \in \rho(A)$.
	
	Now, let $\lambda_1<0$ be in the resolvent set of $A$ and let $c>0$ be a constant such that $S \Res(\lambda_1,A)T \leq c SPT$. Fix $0<f\in E$ and let ${C:=\norm{S\Res(\lambda_1,A)^2T}_{E^{T'\phi}\to G_{Su}}}$. Then
	\[
		\norm{S\Res(\lambda_1, A)^2Tf}_{Su}\leq C \norm{f}_{E^{T'\phi}} =C \duality{T'\phi}{f}.
	\]
	As $A$ is real, so is $\Res(\lambda_1,A)$. In addition, $S$ and $T$ are also real. Hence using the above inequality, we get
	\[
		S\Res(\lambda_1,A)^2Tf \geq - \modulus{S\Res(\lambda_1, A)^2Tf} \geq -C \duality{T'\phi}{f}Su=-C (Su \otimes T'\phi)f. 
	\]
	Therefore
	\begin{equation}\label{eq:resolvent-inequality}
		S\Res(\lambda_1,A)^2T \geq -(c'' C) SPT,
	\end{equation}
	 where $c''>0$ is such that $c'' SPT \geq  (Su \otimes T'\phi)>0$ (obtained by Proposition~\ref{prop:sufficient-projection}).
	 
	 Next, define
	 \[
	 	R_{\lambda}:=S(\lambda_1-\lambda)^2\Res(\lambda_1,A)^2 \lambda \Res(\lambda,A)T
	\]
	for each $\lambda\in \rho(A)$. We claim that ${\lim_{\lambda \to 0}\norm{R_{\lambda}-SPT}_{E^{T'\phi} \to G_{Su}} =0}$. Indeed, since $0$ is simple pole of the resolvent $\Res(\argument,A)$, so the limit ${\lim_{\lambda \to 0}\lambda \Res(\lambda,A) = P}$ exists in the operator norm by~\cite[Proposition~4.3.15]{ArendtBattyHieberNeubrander2011}. Therefore (as in the proof of Theorem~\ref{thm:sufficient-individual-resolvent}), the resolvent equation can be made use of to show that $\linebreak{\lim_{\lambda\to 0} \norm{\lambda\Res(\lambda,A)-P}_{F\to \dom{A}}=0}$. Actually, because $0$ is simple pole of the resolvent $\Res(\argument,A)$, we also have that  $\lambda_1 \Res(\lambda_1, A)P=P$. Let
	\[
		M_1:=\norm{S\Res(\lambda_1,A)}_{\dom{A}\to G_{Su}}\quad \text{and}\quad M_2:=\norm{\Res(\lambda_1,A)T}_{E^{T'\phi} \to F}.
	\]
	Then
	\begin{align*}
		\lim_{\lambda \to 0}\norm{R_{\lambda}-SPT}_{E^{T'\phi} \to G_{Su}} & = \lim_{\lambda \to 0}\norm{R_{\lambda}-S(\lambda_1-\lambda)^2\Res(\lambda_1,A)^2PT}_{E^{T'\phi} \to G_{Su}}\\
												     &\leq  M_1\left(\lim_{\lambda \to 0}(\lambda_1-\lambda)^2 \norm{\lambda\Res(\lambda,A)-P}_{F\to \dom{A}}\right) M_2\\
												     &=0,
	\end{align*}
	which proves our claim.
	
	From the convergence just shown, we infer the existence of a $\delta\in (0,-\lambda_1)$ such that
	 \[
	 	\norm{R_{\lambda}-SPT}_{E^{T'\phi} \to G_{Su}} <\frac{1}{2c''} \text{ for all } -\delta<\lambda<0.
	\]
	Again using that the operators $S,T$, and $A$ are real, we obtain
	\[
		\modulus{R_{\lambda}f-SPTf} < \frac{1}{2c''} \norm{f}_{E^{T'\phi}} Su=\frac{1}{2c''}  \duality{T'\phi}{f} Su =  \frac{1}{2c''} (Su \otimes T'\phi)f
	\]
	for every $0<f\in E$, which gives
	\begin{equation}\label{eq:resolvent-convergence}
		R_{\lambda} > \frac{1}{2} SPT \text{ for all } -\delta<\lambda<0.
	\end{equation}
	
	Finally, let $\lambda \in (-\delta,0)$. Iterating the resolvent equation, we get
	\begin{align*}
	\lambda S \Res(\lambda,A)T&=\lambda S \Res(\lambda_1,A)T + \lambda(\lambda_1-\lambda) S \Res(\lambda_1, A)^2T+ R_{\lambda}\\
					  &> c \lambda  SPT -c''C \lambda(\lambda_1-\lambda) SPT + \frac{1}{2} SPT\\
					  &= \left(c \lambda -c''C \lambda(\lambda_1-\lambda) + \frac{1}{2}\right) SPT;
	\end{align*}
	here the inequality is a result of the assumption ${S\Res(\lambda_1,A)T\leq cSPT}$ along with inequalities in \eqref{eq:resolvent-inequality} and \eqref{eq:resolvent-convergence}. It follows that, for $\lambda$ less than and sufficiently close to $0$ and for all $f\in E_+$, we have
		\[
			\lambda S \Res(\lambda,A)Tf \geq \frac12 SPTf \geq \frac{1}{2c''} (Su \otimes T'\phi) f=c' Su;
		\]
		where $c'= \frac{1}{2c''} \duality{\phi}{Tf}>0$ (due to the strict positivity of $T'\phi$).
\end{proof}

\begin{corollary}[Local uniform maximum principle]
	\label{cor:local-maximum}
	Under the assumptions of Theorem~\ref{thm:local-anti-maximum}, if there exists $\lambda_1$ in the resolvent set of $A$ with ${\lambda_1 >\lambda_0}$ and $c>0$ such that ${S \Res(\lambda_1,A)T \geq -c SPT}$, then there exists $\lambda_2>\lambda_0$ satisfying the following:~${(\lambda_0,\lambda_2]\subseteq \rho(A)}$ and for each $0<f\in E$, there exists a constant $c'>0$ such that ${\lambda S\Res(\lambda,A)Tf \leq c'Su}$ for all $(\lambda_0,\lambda_2]$. In particular, $S\Res(\argument,A)T$ is uniformly eventually strongly positive with respect to $Su$ at $\lambda_0$.
\end{corollary}

\begin{proof}
	As always assume that $\lambda_0=0$. If there exists ${\lambda_1 >0}$ and $c>0$ such that $S \Res(\lambda_1,A)T \geq -c SPT$, then ${S \Res(-\lambda_1,-A)T\leq c SPT}$. The result can now be deduced by applying Theorem~\ref{thm:local-anti-maximum} to the operator $-A$.
\end{proof}

As announced at the beginning of the section, we are now going to prove a criterion sufficient for uniform local eventual strong positivity of semigroups similar to~\cite[Theorem~3.1]{DanersGlueck2018b}. However, we restrict ourselves to semigroups that are \emph{norm continuous at infinity}. This concept was introduced and studied by {Mart\'{\i}nez} and {Maz\'on} in~\cite{MartinezMazon1996} and is weaker as compared to eventual norm continuity. Even so, it has similar spectral and convergence properties (see the aforementioned reference). We recall that a $C_0$-semigroup $(e^{tA})_{t \geq 0}$ on a Banach space $F$ is said to be \emph{norm continuous at infinity} if $\gbd(A) >-\infty$ and
\[
	\lim_{t\to \infty} \limsup_{s\to 0} \norm{e^{t(A-\gbd(A))} - e^{(t+s)(A-\gbd(A))} }=0.
\]

\begin{theorem}
	\label{thm:sufficient-uniform-semigroup}
		Let Assumptions~\ref{ass:standing-assumptions} hold and let $(e^{tA})_{t \geq 0}$ be $C_0$-semigroup which is norm continuous at infinity. Furthermore, suppose the following are fulfilled.
	\begin{enumerate}[\upshape (a)]
		\item The spectral bound $\spb(A)$ is a simple pole of the resolvent $\Res(\argument,A)$ and a dominant spectral value.
		
		\item The eigenspace $\ker(\spb(A)-A)$ is one-dimensional.
		
		\item There exist $t_1,t_2>0$ such that $Se^{t_1 A}F \subseteq G_{Su}$ and $T'e^{t_2A'}F'\subseteq (E')_{T'\phi}$. 
	\end{enumerate}
	
	Let $x$ be a non-zero vector in $\ker(\spb(A)-A)$ and $\psi$ be a non-zero functional in the dual eigenspace $\ker(\spb(A)-A')$ such that $Sx \gg_{Su} 0$ and $T'\psi \gg_{T'\phi} 0$. Then there exists $t_0\geq 0$ such that for each $0<f\in E$, there exists $c>0$ such that $S e^{tA}T f \geq cSu$ for all $t\geq t_0$. In particular, the family $(S e^{tA} T)_{t\geq 0}$ is uniformly eventually strongly positive with respect to $Su$.
\end{theorem}

In Section~\ref{section:applications}, we give explicit examples where assumption (c) in the above theorem holds. 

It is interesting to observe that the constant $c>0$ in the above theorem depends only on the vector $f\in E$ but not on the times $t\geq t_0$. This is consistent with the behaviour exhibited in the previous results of this section as well as with the results in Section~\ref{section:sufficient-conditions-individual}. 

Observe that, unlike Theorem~\ref{thm:sufficient-individual-semigroup-resolvent}, we do not impose any convergence restrictions on the semigroup. In fact, we will show that the assumptions of Theorem~\ref{thm:sufficient-uniform-semigroup} guarantee that the rescaled semigroup $(e^{t(A-\spb(A))})_{t \geq 0}$ converges uniformly as ${t\to \infty}$. For this purpose, we will employ a characterization proved by Thieme in~\cite{Thieme1998} for uniform convergence of the rescaled semigroup. Actually, Thieme used the seemingly weaker notion of \emph{essential norm continuity} instead of norm continuity at infinity. However, these two concepts are equivalent. This was proved by Blake in his PhD thesis~\cite[Corollary~3.3.7]{Blake1999} and also by Nagel and Poland in~\cite[Corollary~4.7]{NagelPoland2000}.

\begin{proof}[Proof of Theorem~\ref{thm:sufficient-uniform-semigroup}]
	Replacing $A$ with $A-\spb(A)$ if necessary, we may assume that $\spb(A)=0$. Since the semigroup is norm continuous at infinity and $0$ is a simple pole of the resolvent $\Res(\argument,A)$ as well as a dominant spectral value, all the conditions sufficient for the uniform convergence of our semigroup given in~\cite[Theorem~2.7]{Thieme1998} are satisfied (cf.~\cite[Definition~2.1 and Proposition~2.3]{Thieme1998}).
	 Let $P$ denote the spectral projection associated to $0$. Since $0$ is a pole of the resolvent $\Res(\argument,A)$, we in fact have that $e^{tA} \to P$ in the operator norm as $t \to \infty$.
	
	Next note that by the closed graph theorem $Se^{t_1A} \in \calL(F,G_{Su})$ and that by Lemma~\ref{lem:operator-extension}, the operator $e^{t_2 A} T$ extends to a bounded operator from $E^{T'\phi}$ to $F$. Therefore for every $t\geq 0$, the operator $S e^{(t+t_1+t_2)A}(I-P)T$ extends to a bounded linear operator from $E^{T'\phi}$ to $G_{Su}$. Let $M_t:=\norm{Se^{(t+t_1+t_2)A}(I-P)T}_{E^{T'\phi} \to G_{Su}}$. Then
	\begin{align*}
		M_t \leq \norm{Se^{t_1A}}_{F \to G_{Su}} \norm{e^{tA}(I-P)}_{F \to F} \norm{e^{t_2A }T}_{E^{T'\phi} \to F}\to 0 \text{ as } t \to \infty.
	\end{align*}
	
	Fix $\epsilon>0$ and $0<f\in E$. Then there exists $t_3>0$ such that $M_t < \epsilon$ for all $t \geq t_3$ and so
	\[
		\norm{Se^{tA}(I-P)Tf}_{Su}\leq M_{t-t_1-t_2} \norm{f}_{E^{T'\phi}} < \epsilon \norm{f}_{E^{T'\phi}} = \epsilon \duality{T'\phi}{f}
	\]
	for all $t \geq t_0:=t_1+t_2+t_3$. Recalling that the operators $S$ and $T$  are real and that our semigroup is also real, we obtain
	\[
		Se^{tA}(I-P)Tf \geq - \modulus{Se^{tA}(I-P)Tf} \geq -\epsilon \duality{T'\phi}{f}Su = -\epsilon (Su \otimes T'\phi)f
	\]
	for all $t\geq t_0$. As $0$ is a simple pole of the resolvent $\Res(\argument,A)$, so $e^{tA}P=P$ and there exist $c'>0$ such that $SPT \geq c' (Su \otimes T'\phi)$ (Proposition~\ref{prop:sufficient-projection}). Whence
	\begin{align*}
		Se^{tA}Tf&=Se^{tA}PTf+ Se^{tA}(I-P)Tf \\
			      &= SPTf+Se^{tA}(I-P)Tf\\
			      & \geq (c'-\epsilon) (Su \otimes T'\phi)f
	\end{align*}
	for all $t \geq t_0$. Choosing $\epsilon=\frac{c'}{2}$ and letting $c=\frac{c'}{2} \duality{T'\phi}{f}>0$, the conclusion follows.
\end{proof}

\section{Applications}
\label{section:applications}

We now illustrate how the sufficient conditions proved in the previous sections can be applied to concrete examples. Of course, all the applications given in~\cite[Section~6]{DanersGlueckKennedy2016a},~\cite[Section~6]{DanersGlueckKennedy2016b},~\cite[Section~4]{DanersGlueck2018b}, and~\cite[Chapter~11]{Glueck2016} are particular examples of local eventual positivity with $E=F=G$ and $S$ and $T$ as the identity operator. However, we will show that our setting gives us more freedom on domains when dealing with function spaces.

\subsection*{The Dirichlet bi-Laplacian}

Let $\Omega \subseteq \bbR^d$ be a non-empty bounded domain with $C^{\infty}$-boundary and fix $p \in (1,\infty)$. On $F:=L^p(\Omega,\bbC)$, we consider the \emph{bi-Laplace operator with Dirichlet boundary conditions} $A_p: \dom{A_p} \subseteq F \to F$ given by 
\begin{equation*}
	\begin{split}
		\dom{A_p} &=W^{4,p}(\Omega) \cap W_0^{2,p}(\Omega)\\
		  A_p f &= -\Delta^2 f.
	\end{split}
\end{equation*}
For historical remarks and other details about the Dirichlet bi-Laplacian, we refer the reader to~\cite{GrunauSweers1996,GrunauSweers1997a,GrunauSweers1997b}.

The following properties of $A_p$ are well-known:~the operator $A_2$ is self-adjoint and $A_p$ has compact resolvent. Moreover, $A_p$ generates an analytic $C_0$-semigroup on $F$. In addition, the spectrum $\spec(A_p)$ consists only of eigenvalues of $A_p$, the corresponding eigenspaces are independent of $p$, and the resolvent operators are consistent on the $L^p$-scale. The details can be found, for instance in~\cite[Proposition~8.3.1]{Glueck2016}.

Furthermore, if the domain is sufficiently close to a unit ball in $\bbR^d$ (in the sense of~\cite[Theorem~5.2]{GrunauSweers1997b}), then due to the concrete results obtained by Grunau and Sweers (in the aforementioned reference), the abstract theory of eventually positive semigroups could be applied to the Dirichlet bi-Laplacian. As a result, individual eventual positivity properties of both the semigroup and the resolvent were proved in~\cite[Section~11.4]{Glueck2016} and uniform eventual positivity of the semigroup was shown in~\cite[Theorem~4.4]{DanersGlueck2018b}.  In contrast, our setting does not require any additional assumption on the domain. While on one hand, this allows us to only show individual local eventual positivity of the resolvent for sufficiently large $p$, on the other, we obtain uniform local eventual positivity of the semigroup for all $p\in (1,\infty)$.

Assume the eigenspace $\ker(\spb(A_p)-A_p)$ is generated by a vector $x$ and let $V \subseteq \Omega$ be a connected component of $[x \neq 0]$. Replacing $x$ with $-x$ if necessary, we may assume $x>0$ on $V$. Fix a compact subset $K$ of $V$ and define ${S:F \to F}$ by $S f:= \one_K f$ for all $f\in F$, where $\one_K$ denotes the indicator function of $K$. Intuitively, the operator $S$ corresponds to restricting the functions to a compact subset of a \emph{nodal} domain. Next, let $E=\{f\restrict{K}: f\in F\}$ and $T\in \calL(E,F)$ denote the extension by zero operator. It is clear that both $S$ and $T$ are positive operators and ${T': F'\to E'}$ is given by $T' g = \one_K g$ for all $g \in F'=L^q(\Omega)$; here $q$ is such that $1/p+1/q=1$. In particular, $T'\one_\Omega=\one_K$ is a strictly positive functional on $E'$. 

\begin{theorem}
	The family $\left(S e^{tA_p} T\right)_{t \geq 0}$ is uniformly eventually strongly positive with respect to $\one_K$ and if $p> \frac{d}{2}$, then $S\Res(\argument,A)T$ is individually eventually strongly positive and negative with respect to $\one_K$ at $\spb(A_p)$.
\end{theorem}

\begin{proof}
	 Since $A_p$ has compact resolvent, therefore $\spb(A_p)$ is a pole of the resolvent $\Res(\argument,A_p)$. As in the proofs of~\cite[Proposition~8.3.1 and Corollary~8.3.3]{Glueck2016}, it can be shown that the dual operator $A_p'$ coincides with the operator $A_q$ (where $q$ is such that $1/p+1/q=1$) and $x \in \ker(\spb(A_p)-A_p')$, as well. When talking about $x$ as a functional on $F$, we will denote it by $\psi$.
	 
	 Note that $\dom{A_p^n} \subseteq W^{4n,p}(\Omega)$ for each $n\in\bbN$ by~\cite[Corollary~2.21]{GazzolaGrunauSweers2010}. We now fix $n\in\bbN$ large enough, so that the Sobolev embedding theorem and the definition of $\dom{A_p}$ can together guarantee ${\dom{A_p^n} \subseteq C^2\left(\overline{\Omega}\right)\cap W_0^{2,p}(\Omega)}$.  Combining this with the analyticity of the semigroup implies that
	 \[
	 	e^{tA_p}F \subseteq \bigcap_{m\in\bbN} \dom{A_p^m} \subseteq \dom{A_p^n} \subseteq  L^{\infty}(\Omega).
	\]
	for each $t> 0$. Consequently, ${Se^{tA_p}F \subseteq F_{\one_K}}$. As $A_p'$ coincides with $A_q$, so the inclusion ${T'e^{tA_p'}F' \subseteq (E')_{\one_K}}$ holds as well. Now let $B_p:=A_p-\spb(A_p)$. Then the semigroup $\left(e^{tB_p}\right)_{t\geq 0}$ is bounded on $L^{2}(\Omega)$. Also,
	\[
		\norm{e^{tB_p}}_{F\to F}\leq \norm{I}_{L^{\infty}\to F} \norm{e^{B_p}}_{L^{2}\to L^{\infty}} \norm{e^{(t-2)B_p}}_{L^{2}\to L^{2}}\norm{I}_{L^{\infty}\to L^{2}}\norm{e^{B_p}}_{F\to L^{\infty}}
	\]
	for each $t\geq 2$; where $I$ denotes the corresponding identity operator. It follows that the semigroup $\left(e^{tB_p}\right)_{t\geq 0}$ is also bounded on $F$. From this, we can infer that the spectral bound $\spb(A_p)$ is a simple pole of the resolvent $\Res(\argument,A_p)$. Further, all spectral values of $A_2$ are real because the operator is self-adjoint. This means that $\spb(A_p)$ must be a dominant spectral value of $A_p$. In addition, we also have that $Sx \gg_{\one_K} 0$ and $T'\psi \gg_{\one_K} 0$. As every analytic semigroup is norm continuous at infinity, we can apply Theorem~\ref{thm:sufficient-uniform-semigroup} to conclude that $\left(S e^{tA_p} T\right)_{t \geq 0}$ is uniformly eventually strongly positive with respect to $\one_K$.
	 
	 Finally, let $p>\frac{d}2$. Then by the Sobolev embedding theorem
	 	\[
	 		\dom{A_p} = W^{4,p}(\Omega) \cap W_0^{2,p}(\Omega) \subseteq C^2\left(\overline{\Omega}\right)\cap W_0^{2,p}(\Omega)
		\]
	 and so $S\dom{A_p} \subseteq F_{\one_K}$.  We can now deduce from Theorem~\ref{thm:sufficient-individual-resolvent} that $S \Res(\argument,A) T$ is individually eventually strongly positive and negative with respect to $\one_K$ at $\spb(A_p)$.
\end{proof}

\subsection*{Square of the Dirichlet Laplacian}

The semigroup generated by the square of the Dirichlet Laplacian on $C_0(\Omega)$ (with sufficiently smooth $\Omega$) was shown to be individually eventually strongly positive with respect to the function $\dist{\argument}{\partial \Omega}$ in ~\cite[Theorem~6.1]{DanersGlueckKennedy2016b} (see also~\cite[Theorem~11.5.1]{Glueck2016}). Here we study the local eventual positive behaviour of the semigroup associated to this operator, but we allow rough domains and instead consider the space $L^2(\Omega)$.

Let $\Omega \subseteq \bbR^{d}$ be a non-empty bounded domain and denote by $\Delta_D$ the \emph{Laplace operator with Dirichlet boundary conditions} on $F:= L^2(\Omega)$, also known as the \emph{Dirichlet Laplacian}, i.e.,
\begin{equation*}
	\begin{split}
		\dom{\Delta_D} &=\{u \in H_0^1(\Omega):\Delta u \in F\}\\
		  \Delta_D f &= \Delta f.
	\end{split}
\end{equation*}
The operator $\Delta_D$ is associated with the form
\[
	a(u,v)=\int_\Omega \nabla u \nabla v
\]
on $H_0^1(\Omega)$. Employing form methods, it can be shown that $\Delta_D$ is self-adjoint and generates a contraction analytic $C_0$-semigroup of angle $\pi/2$ on $F$ and that ${\spec(\Delta_D) \subseteq (-\infty,0)}$;~see~\cite[Theorem~4.18]{ArendtChillSeifertVogtVoigt2015} and~\cite[Section~1]{ArendtBenilan1999}. As in the proof of~\cite[Proposition~6.6]{DanersGlueckKennedy2016a}, it can also be shown that $A:=-\Delta_D^2$ also generates a bounded analytic $C_0$-semigroup of angle $\pi/2$ on $F$.

It is well-known that $\spb(A)<0$ is a simple pole of the resolvent $\Res(\argument,A)$, a dominant spectral value, and that there exists a vector $x \in F$ of fixed sign such that $\ker(\spb(A)-A)$ is spanned by $x$; see~\cite[Theorem~10.4.1(i) and Exercise~11.6.7]{Strauss2007}. Without loss of generality, we assume $x>0$. Since $F$ is a Hilbert lattice, the Banach space dual of $A$ coincides with the Hilbert space adjoint of $A$, i.e., ${A'= A^*=A}$. This implies that $\ker(\spb(A)-A')$ contains a strictly positive functional, namely $x$.

Let $K$ be any compact subset of $\Omega$ and let $S$ be the positive operator given by $S: F \to F$ by $Sf:= \one_K f$ for all $f \in F$. From the Harnack inequality~\cite[Theorem~8.20]{GilbrargTrudinger2001}, we infer $x(\omega)>0$ for every $\omega\in \Omega$, which readily gives $Sx \gg_{\one_K} 0$. Furthermore, because there exists $n\in \bbN$ such that $\dom{A^n} \subseteq L^{\infty}(\Omega)$~\cite[Page~33]{ArendtBenilan1999}, whence as in Remark~\ref{rem:sufficient-individual}(b), there exists $t_0 >0$ such that
\[
	Se^{t_0 A}F \subseteq F_{\one_K} \text{ and } e^{t_0 A'}F' \subseteq (F')_{\one_K}.
\]
 Lastly, noting that as $(e^{tA})_{t \geq 0}$ is analytic, it is in particular, norm continuous at infinity and so by Theorem~\ref{thm:sufficient-uniform-semigroup}, we have proved the following:
 
\begin{theorem}
	\label{thm:square-dirichlet-laplacian}
 	Let $\Omega \subseteq \bbR^{d}$ be a non-empty bounded domain and $\Delta_D$ be the Dirichlet Laplacian on $L^2(\Omega)$. Let $K$ be any compact subset of $\Omega$ and $S$ be the operator on $L^2(\Omega)$ given by multiplication with the indicator function of $K$. Then $A:=-\Delta_D^2$ generates a bounded analytic $C_0$-semigroup of angle $\pi/2$ and $(S e^{tA})_{t\geq 0}$ is uniformly eventually strongly positive with respect to $\one_K$.
\end{theorem}
 
 Since we did not impose any smoothness conditions on the boundary of $\Omega$, one cannot expect that the vectors in $e^{t_0 A} F$ are dominated by a multiple of the first eigenfunction. For this reason, the theory of eventual positivity in~\cite{DanersGlueckKennedy2016b} or~\cite{DanersGlueck2018b} is not applicable. However, the operator $S$ ensures that we stay away from the boundary and so the \emph{smoothing condition} is satisfied in our setting. This is why we were able to apply our abstract results even on rough domains.
 
\subsection*{A diffusion equation with non-autonomous boundary}

In this example, we consider a diffusion process on an interval equipped with non-autonomous boundary conditions. This model was considered in~\cite[Section~8.6]{Glueck2016} where \emph{asymptotic positivity} of the associated semigroup was shown.

Fix $r>0$ and let $f_b:[0,\infty) \to \bbC$ be a given function supported on $[0,r]$. Further, let $\beta_1,\beta_2\geq 0$ be such that $\beta_1+\beta_2 \leq 1$. For $f(0,\argument) \in L^1((0,1),\bbC)$, consider
\begin{equation*}
	\begin{split}
		\partial_t f(t,x)&=\partial_x^2 f(t,x), \quad x \in (0,1)\\
		\partial_x f(t,0)&=-\beta_1 f_b(t)\\
		\partial_x f(t,1)&=-\beta_2 f_b(t)
	\end{split}
\end{equation*}
on $L^1((0,1),\bbC)$. Then using methods similar to~\cite[Section~VI.7]{EngelNagel2000}, we can reformulate this as an abstract Cauchy problem
\[
	\frac{d}{dt}(f,f_b)=A(f,f_b)
\]
on $F:= L^1((0,1),\bbC) \times L^1((0,r),\bbC)$, where
\begin{equation}\label{eq:diffusion-domain}
	\begin{split}
		\dom{A}&= \left\{(f,f_b) \in W^{2,1}((0,1)) \oplus W^{1,1}((0,r)): \begin{array}{c}
														f'(0)=-\beta_1 f_b(0),\\
														 f'(1)=-\beta_2 f_b(0),\\
														 f_b(r)=0
												\end{array}\right\}\\
		A(f,f_b)&= (f'',f_b').
	\end{split}
\end{equation}
We endow $F$ with the norm $\norm{(f,f_b)}:=\norm{f}_{L^1((0,1),\bbC)}+\norm{f_b}_{L^1((0,r),\bbC)}$ for all ${(f,f_b) \in F}$. The following was proved in~\cite[Theorem~8.6.1]{Glueck2016}:

\begin{theorem}
	\label{thm:diffusion-gluck}
	The operator $A$  generates a contraction $C_0$-semigroup on $F$. Moreover, we have the following:
	\begin{enumerate}[\upshape (a)]
		\item The operator $A$ is dissipative and has compact resolvent.
		
		\item The spectrum of $A$ is given by $\{-k^2\pi^2 : k \in \bbN \cup \{0\}\}$.
		
		\item Both the eigenspaces $\ker A$ and $\ker A'$ are one-dimensional and contain the vector $(\one_{(0,1)},0)$ and the functional $(\one_{(0,1)},(\beta_1-\beta_2)\one_{(0,r)})$ respectively.
		
		\item The limit
			\[
				\lim_{t\to\infty}\dist{e^{tA}(f,f_b)}{F_+} = 0 \text{ for every } (f,f_b) \geq 0
			\]
 if and only if $\beta_1 \geq \beta_2$.
		
		\item The semigroup $(e^{tA})_{t\geq 0}$ is positive if and only if $\beta_2=0$.
	\end{enumerate}
\end{theorem}

The assertion (d) in the above theorem says that the semigroup $(e^{tA})_{t\geq 0}$ is \emph{individually asymptotically positive} (cf.~\cite[Definition~8.1]{DanersGlueckKennedy2016b}) if and only if ${\beta_1 \geq \beta_2}$. We show that for ${\beta_1 > \beta_2}$, the semigroup is eventually positive on the first component of the space $F$.

\begin{theorem}
	Let $F:= L^1((0,1),\bbC) \times L^1((0,r),\bbC)$ and $\beta_1>\beta_2 \geq 0$ be such that $\beta_1+\beta_2\leq 1$. Denote by $S: F \to L^1((0,1),\bbC)$, the projection onto the first component of $F$. Further, let $A$ be the operator given by \eqref{eq:diffusion-domain}. Then $S\Res(\argument,A)$ is individually strongly positive with respect to $\one_{(0,1)}$ at $\spb(A)$ and the family $(S e^{tA})_{t \geq 0}$ is individually eventually strongly positive with respect to $\one_{(0,1)}$.
\end{theorem}

\begin{proof}
	Let
	\[
		x :=(\one_{(0,1)},0) \in F\quad \text{and}\quad \psi :=(\one_{(0,1)},(\beta_1-\beta_2)\one_{(0,r)}) \in F'.
	\]
	By Theorem~\ref{thm:diffusion-gluck}, the operator $A$ has compact resolvent and $\spb(A)=0\in\spec(A)$. Therefore, $0$ is a pole of the resolvent $\Res(\argument,A)$. Moreover, the corresponding semigroup is bounded and both the spectrum of $A\restrict{\ker \psi}$ and the point spectrum of $\left(A\restrict{\ker \psi}\right)'$ do not intersect the imaginary axis. Hence in view of the ABLV theorem~\cite[Theorem~V.2.21]{EngelNagel2000}, we have that $e^{tA} \restrict{\ker \psi}$ converges strongly to $0$ as $t\to \infty$. Now $x \in \ker A$, the vector $Sx$ is strongly positive with respect to $\one_{(0,1)}$, and $\psi\in \ker A'$ is a strictly positive functional. Furthermore,
	\[ 
		S\dom{A} \subseteq W^{2,1}((0,1)) \subseteq C([0,1]) \subseteq L^{\infty}((0,1)) = F_{\one_{(0,1)}}.
	\]
	Accordingly, all the assumptions of Corollary~\ref{cor:sufficient-individual-semigroup-resolvent} are satisfied, which means that $S\Res(\argument,A)$ must be individually strongly positive with respect to $\one_{(0,1)}$ at $0$.
	
	To conclude that $(S e^{tA})_{t \geq 0}$ is individually eventually strongly positive with respect to $\one_{(0,1)}$, it suffices by Corollary~\ref{cor:sufficient-individual-semigroup-resolvent}, to show that there exists $t_0 >0$ such that ${S e^{t_0 A} \subseteq F_{\one_{(0,1)}}}$. To this end, we let $Y$ be the subspace $L^1((0,1)) \times \{0\}$ of $F$ and let $T$ denote the projection onto the second component of $F$. A simple argument shows that for $\lambda \in \rho(A)$, we have $T\Res(\lambda, A)(f,0)=0$ for every $(f,0)\in F$, and so $\Res(\lambda, A)$ leaves $Y$ invariant. We can now make use of the Post-Widder Inversion Formula ~\cite[Corollary~III.5.5]{EngelNagel2000} to infer that $Y$ is an $(e^{tA})_{t\geq 0}$-invariant closed subspace of $F$. Thus by~\cite[Paragraph~II.2.3]{EngelNagel2000}, we get that the generator of the restricted semigroup $\left(e^{tA} \restrict{Y}\right)_{t \geq 0}$ is given by 
\begin{equation*}
	\begin{split}
		D(A \restrict{Y}) &= \left\{(f,0) \in W^{2,1}((0,1)) \times \{0\}:f'(0)=f'(1)=0\right\}\\
		A \restrict{Y}(f,0)&=(f'',0).
	\end{split}
\end{equation*}
Hence if $\Delta_N$ denotes the Neumann-Laplace operator on $L^1((0,1))$, then
\begin{equation}
	\label{eq:smoo}
	S e^{tA} \restrict{Y}=e^{t\Delta_N}S\quad \text{for every }t \geq 0.
\end{equation}

Applying the arguments of~\cite[Paragraph~II.2.4]{EngelNagel2000} together with the observation $X/Y \cong \Ima T$, we obtain that the generator of the quotient semigroup $\left(e^{tA}/_Y\right)_{t \geq 0}$ is
\begin{equation*}
	\begin{split}
		D(A/_Y) &= \left\{f_b \in W^{1,1}((0,r)):f_b(r)=0\right\}\\
		A/_Y f_b&=f_b'.
	\end{split}
\end{equation*}
In other words, if $(e^{tB})_{t \geq 0}$ denotes the left shift semigroup on $ L^1((0,r))$, then $T e^{tA}=e^{tB}T$ for every $t \geq 0$. In particular, $Te^{rA} =e^{rB}T=0$ and so for $(f,f_b) \in F$, there exists $(g,0)\in Y$ such that $(g,0)=e^{rA}(f,f_b)$. Whence for every $t \geq 0$,
\[
	Se^{(t+r)A}(f,f_b)=Se^{tA}(g,0)=e^{t\Delta_N}g
\]
by \eqref{eq:smoo} and 
\[
	e^{t\Delta_N}Se^{rA}(f,f_b)=e^{t\Delta_N}S(g,0)=e^{t\Delta_N}g.
\]
Combining these yields
\[
	S e^{(t+r)A} = e^{t \Delta_N} S e^{rA},
\]
for every $t\geq 0$ which implies
\[
	S e^{(t+r)A} F \subseteq e^{t \Delta_N} L^1((0,1)) \subseteq L^{\infty}((0,1)) = F_{\one_{(0,1)}}
\]
for every $t>0$, as required.
\end{proof}

\section{Consequences of local eventual positivity I}
\label{section:consequences-i}

The rest of this article is devoted to proving the implications of local eventual positivity. While the characterizations of individually eventually positive semigroups given in~\cite[Section~5]{DanersGlueckKennedy2016b} provide various consequences of eventual positivity, several other properties were already proved in~\cite[Section~7]{DanersGlueckKennedy2016a}. We show that similar spectral properties are exhibited by locally eventually positive semigroups as well.

Recall from~\cite[Theorem~II.1.10]{EngelNagel2000} that, if $(e^{tA})_{t \geq 0}$ is a $C_0$-semigroup on a Banach space $F$, then the resolvent $\Res(\argument,A)$ can be represented as the Laplace transform of the semigroup on the half-plane to the right of $\gbd(A)$, i.e.,
\begin{equation}\label{eq:laplace}
	\Res(\lambda,A)f = \int_0^{\infty} e^{-\lambda t} e^{tA} f \, dt
\end{equation}
for all $f\in F$ and $\re \lambda>\gbd(A)$, where the integral exists as an improper Riemann integral. If $F$ is, in addition, a Banach lattice and the semigroup is positive, then it is well-known that \eqref{eq:laplace} is also valid for $\re \lambda >\spb(A)$ (see~\cite[Proposition~5.1.4 and Theorem~5.3.1]{ArendtBattyHieberNeubrander2011}). Moreover, it was proved in~\cite[Proposition~7.1]{DanersGlueckKennedy2016a} that this remains true when the semigroup is merely individually eventually positive. We now prove a version of this in our setting.

Let $E,F$, and $G$ be complex Banach lattices and let $(e^{tA})_{t \geq 0}$ be a real $C_0$-semigroup on $F$. If $S \in \calL(F,G)$ and $T\in \calL(E,F)$ are positive operators, then we call the family $(S e^{tA} T)_{t\geq 0}$ is individually eventually positive, if for every ${f\in E_+}$, there exists $t_0\geq 0$ such that $Se^{tA}T f \geq 0$ for all $t\geq t_0$. 

\begin{proposition}
	\label{prop:laplace-representation}
	Let $E,F$, and $G$ be complex Banach lattices and let $(e^{tA})_{t \geq 0}$ be a real $C_0$-semigroup on $F$. If $S \in \calL(F,G)$ and $T\in \calL(E,F)$ are positive operators such that the family $(S e^{tA} T)_{t\geq 0}$ is individually eventually positive, then 
	\[
		S \Res(\lambda,A) Tf = \int_0^{\infty} e^{-\lambda t} Se^{tA}Tf\, dt
	\]
for all $\re \lambda >\spb(A)$ and $f \in E$, where the integral converges as an improper Riemann integral.
\end{proposition}

\begin{proof}
	Let $f \in E$. Since the positive cone of a Banach lattice is generating, we may assume $f\geq 0$. Then due to the individual eventual positivity of $(S e^{tA} T)_{t\geq 0}$, there exists $t_0 \geq 0$ such that $Se^{tA}Tf \geq 0$ for all $t \geq t_0$. Define $v:[0,\infty)\to E$ by ${v(t)=S e^{(t+t_0)A}Tf}$ for all $t\geq 0$. As $v$ is a positive function, so by~\cite[Theorem~1.5.3]{ArendtBattyHieberNeubrander2011}, the abscissa of holomorphy of its Laplace integral $\widehat{v}$ is equal to its abscissa of convergence, i.e., $\hol\left(\widehat{v}\right)=\abs(v)$. Thus, using the identity theorem for analytic functions, we obtain
	\[
		S\Res(\lambda,A)Tf=\int_0^{t_0} e^{-\lambda t} Se^{tA}Tf\, dt + e^{-t_0 \lambda} \widehat{v}(\lambda) \text{ for all } \re \lambda > \abs(v).
	\]
	
	Now if $\abs(v) >\spb(A)$, then $S\Res(\argument,A)Tf$ will have a singularity at $\abs(v)$, again due to~\cite[Theorem~1.5.3]{ArendtBattyHieberNeubrander2011}. This would further imply that $\Res(\argument,A)Tf$ has a singularity at $\abs(v)$, as $S$ is a bounded operator. Therefore, in this case $\abs(v)$ would be a spectral value of $A$ which is a contradiction to $\abs(v) > \spb(A)$. Consequently, we must have $\abs(v)\leq \spb(A)$, which leads to the assertion.
\end{proof}

The following generalization of~\cite[Corollary~7.3]{DanersGlueckKennedy2016a} now follows in a facile manner.

\begin{corollary}
	\label{cor:laplace-representation}
	Let $E,F$, and $G$ be complex Banach lattices and let $(e^{tA})_{t \geq 0}$ be a real $C_0$-semigroup on $F$. Further, let $S \in \calL(F,G)$ and $T\in \calL(E,F)$ be positive operators such that the family $(S e^{tA} T)_{t\geq 0}$ is individually eventually positive. If $\spb(A)>-\infty$, then
	\[
		\lim_{\lambda \downarrow \spb(A)} (\lambda-\spb(A)) \dist{S\Res(\lambda,A)Tf}{G_+} =0
	\]
	for every $f \in E_+$.
\end{corollary}

In~\cite[Theorem~7.6]{DanersGlueckKennedy2016b}, it was shown that \emph{asymptotic positivity} of the resolvent of a closed operator at a simple pole is equivalent to positivity of the  spectral projection 
associated to that pole. A version of this can be formulated in our current setting. The proof is similar, but for the convenience of the reader, we include it.

\begin{proposition}
	\label{prop:characterization-asymptotic-resolvent}
	Let $E,F$, and $G$ be complex Banach lattices and let $\lambda_0$ be a real spectral value of a closed, densely defined, and real operator $A$ on $F$ and also a simple pole of the resolvent $\Res(\argument,A)$. Further, let $S \in \calL(F,G)$ and $T\in \calL(E,F)$ be positive operators. The following are equivalent.
	\begin{enumerate}[ref=(\roman*)]
		\item\label{prop:characterization-asymptotic-resolvent:projection} The spectral projection $P$ associated to $\lambda_0$ satisfies $SPT \geq 0$.
		
		\item\label{prop:characterization-asymptotic-resolvent:asymptotic-positive} For each $f \in E_+$, we have $(\lambda-\lambda_0)\dist{S\Res(\lambda,A) Tf}{G_+} \to 0$ as $\lambda \downarrow \lambda_0$.
		
		\item\label{prop:characterization-asymptotic-resolvent:asymptotic-bounded} There exist $\lambda_1>\lambda_0$ and $C \geq 0$ such that $(\lambda_0,\lambda_1]\subseteq \rho(A)$ and\linebreak ${\dist{S\Res(\lambda,A)T f}{G_+}\leq C \norm{f}}$ for all $f\in E_+$ and $\lambda \in (\lambda_0,\lambda_1]$.
	\end{enumerate}
\end{proposition}

\begin{proof}
	Of course \ref{prop:characterization-asymptotic-resolvent:asymptotic-bounded} implies \ref{prop:characterization-asymptotic-resolvent:asymptotic-positive}. Moreover, as $\lambda_0$ is a simple pole of the resolvent $\Res(\argument,A)$, we have
	\[
		\dist{SPTf}{G_+} = \lim_{\lambda \downarrow \lambda_0} \dist{S(\lambda-\lambda_0)\Res(\lambda,A)Tf}{G_+}
	\]
	for every $f \in E$ and whence \ref{prop:characterization-asymptotic-resolvent:asymptotic-positive} implies \ref{prop:characterization-asymptotic-resolvent:projection}. For the remainder of the proof, assume without loss of generality that $\lambda_0=0$. 
	
	\Implies{prop:characterization-asymptotic-resolvent:projection}{prop:characterization-asymptotic-resolvent:asymptotic-bounded} Since $0$ is a pole of the resolvent $\Res(\argument,A)$, we can choose $\epsilon>0$ such that $(0,\epsilon] \subseteq \rho(A)$. In addition, as the order of $0$ as a pole of the resolvent $\Res(\argument,A)$ is $1$, we have ${\lambda \Res(\lambda,A)P =P}$ which gives that
	\[
		S\Res(\lambda,A)T=\lambda^{-1}SPT+S\Res(\lambda,A)(I-P)T
	\]
	for all $\lambda \in \rho(A)$. Using ${SPT \geq 0}$, we obtain
	\[
		\dist{S\Res(\lambda,A)Tf}{G_+} \leq \norm{S\Res(\lambda,A)(I-P)Tf}
	\]
	for all $f \in E_+$ and $\lambda \in (0,\epsilon)$. Setting
	\[
		C:=\sup_{\lambda \in (0,\epsilon)} \norm{S \Res(\lambda,A) (I-P) T} <\infty,
	\]
	the conclusion follows.
\end{proof}

Various Kre\u{\i}n-Rutman type theorems about the existence of  positive eigenvectors (as a result of eventual positivity) have been proved in~\cite{DanersGlueckKennedy2016a,Glueck2017,DanersGlueck2017}. For instance, it was shown in~\cite[Theorem~7.7(i)]{DanersGlueckKennedy2016a} that if $(e^{tA})_{t\geq 0}$ is an individually eventually positive semigroup on a complex Banach lattice such that $\spb(A)>-\infty$ is a pole of the resolvent $\Res(\argument,A)$, then it is an eigenvalue of $A$ and the corresponding eigenspace contains a positive eigenvector. In~\cite[Section~6]{Glueck2017}, similar results were proved for both eventually and asymptotically positive operators. We end this section with a similar result along the lines of~\cite[Theorem~3.1]{DanersGlueck2017}, which in particular says:~suppose $A$ is a closed and densely defined linear operator on a complex Banach lattice $F$ and there exists a real spectral value $\lambda_0$ of $A$, which is also a pole of the resolvent. If ${(\lambda-\lambda_0)\dist{\Res(\lambda,A)f}{F_+}\to 0}$ as $\lambda \downarrow \lambda_0$, then $\lambda_0$ is an eigenvalue of both $A$ and $A'$ and the corresponding eigenspaces contain a positive, non-zero vector.

\begin{proposition}
	\label{prop:sufficient-existence-vectors}
	Let $E,F$, and $G$ be complex Banach lattices and let  $\lambda_0$ be a real spectral value of a closed, densely defined, and real operator $A$ on $F$ such that $\lambda_0$ is also a pole of the resolvent $\Res(\argument,A)$ of order, say $m\in \bbN$. Further, let $S \in \calL(F,G)$ and $T\in \calL(E,F)$ be positive operators such that for each $f\in E_+$, we have 
	\begin{equation}\label{eq:asymptotic-resolvent}
		\lim_{\lambda \downarrow \lambda_0}(\lambda-\lambda_0)\dist{S\Res(\lambda,A) Tf}{G_+} = 0.
	\end{equation}
	 Denote by $U_{-m}\in \calL(F)$, the coefficient of $(\lambda-\lambda_0)^{-m}$ in the Laurent series expansion of $\Res(\argument,A)$ about $\lambda_0$. If $SU_{-m}T \neq 0$, then there exists a vector $x$ in the eigenspace $\ker(\lambda_0-A)$ and a functional $\psi$ in the dual eigenspace $\ker(\lambda_0-A')$ such that both $Sx$ and $T'\psi$ are positive and non-zero.
\end{proposition}

\begin{proof}
	Since $\lambda_0$ is a pole of the resolvent $\Res(\argument,A)$, it is an eigenvalue of both $A$ and $A'$ with $\Ima U_{-m} \subseteq \ker(\lambda_0-A)$ and $\Ima U_{-m}' \subseteq \ker(\lambda_0-A')$. Moreover, ${(\lambda-\lambda_0)^m \Res(\lambda,A)\to U_{-m}}$ in the operator norm topology as $\lambda \downarrow \lambda_0$ (all these results can be found, for instance in~\cite[Paragraph~IV.1.7]{EngelNagel2000}). 
	
	Therefore we also have ${(\lambda-\lambda_0)^m S\Res(\lambda,A)T\to SU_{-m}T}$ with respect to the operator norm as $\lambda \downarrow \lambda_0$. Hence if $SU_{-m}T \neq 0$, then because of \eqref{eq:asymptotic-resolvent} and the previous convergence, it is both positive and non-zero. Since range of a positive non-zero operator always contains a positive non-zero vector, we have the existence of a vector ${x \in U_{-m}T \subseteq \ker(\lambda_0-A)}$ such that $Sx >0$. The second assertion follows immediately by considering the positive non-zero operator $T'U_{-m}' S'=(SU_{-m}T)'$ and using ${\Res(\argument,A')=\Res(\argument,A)'}$.
\end{proof}

\begin{remark}
	\label{rem:sufficient-existence-vectors}
	In comparison to~\cite[Theorem~3.1]{DanersGlueck2017}, we needed in Proposition~\ref{prop:sufficient-existence-vectors} the additional assumption that $SU_{-m}T \neq 0$ to prove the existence of positive eigenvectors, as this may not be true in general. However, assume that $E=F=G$ and let $(S_n)$ and $(T_n)$ be sequences of band projections on $E$, converging strongly to the identity operator such that \eqref{eq:asymptotic-resolvent} holds for each $S_n$ and $T_n$, i.e., 
	\[
		\lim_{\lambda \downarrow \lambda_0} (\lambda-\lambda_0)\dist{S_n \Res(\lambda,A)T_n f}{E_+}=0
	\]
	for each $f\geq 0$ and $n\in \bbN$. Then proceeding as in the proof of Proposition~\ref{prop:sufficient-existence-vectors}, it can be shown that $S_n U_{-m} T_n$ is positive for large $n \in \bbN$ and consequently $U_{-m}$ is positive. Therefore there exists a positive non-zero vector $v \in E$ such that $x:=U_{-m}v \in \ker(\lambda_0-A)$ is positive and non-zero. By the same reasoning, $\ker(\lambda_0-A')$ contains a positive non-zero functional. This observation will play an crucial role in Section~\ref{section:consequences-ii}.
\end{remark}

\section{Consequences of local eventual positivity II}
\label{section:consequences-ii}

In this last section, we continue proving the implications of local eventual positivity but obtain stronger conclusions. For instance, we proved the existence of eigenvectors in Proposition~\ref{prop:sufficient-existence-vectors}, but could only talk about the positivity of the vector $Sx$. However, a stronger assertion would be strong positivity of $x$. Naturally, one cannot expect that strong positivity of $S\Res(\argument,A)T$ or of a family $(S e^{tA} T)_{t\geq 0}$ implies existence of strongly positive eigenvectors of $A$. However, Remark~\ref{rem:sufficient-existence-vectors} suggests that if we have strong positivity of $S_n\Res(\argument,A)T_n$ or the families $(S_n e^{tA} T_n)_{t\geq 0}$, where $(S_n)$ and $(T_n)$ are sequences of band projections converging strongly to the identity operator, then one can expect stronger assertions.

Let $\Omega \subseteq \bbR^d$ be a non-empty bounded domain and let $\Delta_D$ be the Dirichlet Laplacian on $L^2(\Omega)$. Consider a sequence of compact subsets $K_n\subseteq K_{n+1}\subseteq \Omega$ such that $\Omega = \bigcup_{n\in\bbN} K_n$. For each $n\in \bbN$, let $S_n \in \calL(L^2(\Omega))$ be defined as $S_n f :=\one_{K_n}f$ for all $f\in \calL(L^2(\Omega))$. Then $(S_n)$ is a sequence of band projections on $L^2(\Omega)$ converging strongly to the identity operator and the family $(S_ne^{-t\Delta_D^2})_{t\geq 0}$ is eventually strongly positive with respect to $\one_{K_n}$ (Theorem~\ref{thm:square-dirichlet-laplacian}). Let $x$ be an eigenvector corresponding to $\spb(-\Delta_D^2)$. Then using $e^{-t\Delta_D^2}x = e^{t \spb(-\Delta_D^2)} x$ for all $t\geq 0$, it follows that if $x$ is positive, then it is, in fact, strongly positive with respect to $\one_{K_n}$ for each $n\in\bbN$. As a consequence, $x$ is a quasi-interior point of $L^2(\Omega)$. Whereas this argument does not provide us with any new information, it does give us an intuition why considering a sequence of band projections, as mentioned above, is the right setting to obtain meaningful results.

Recall that a projection $S$ on a Banach lattice is called a \emph{band projection} if both $S$ and $I-S$ are positive. Intuitively, band projections correspond to picking a part of a function that is supported on a certain subset of the domain in a $\sigma$-finite measure space $(\Omega,\mu)$. For example, if $p\in [1,\infty]$, then by Riesz decomposition theorem (see~\cite[Theorem~II.2.10]{Schaefer1974} or~\cite[Theorem~1.2.9]{Meyer-Nieberg1991}), a projection $S$ is a band projection on $L^p(\Omega,\mu)$ if and only if there exists a measurable set $A \subseteq \Omega$ such that ${\Ima S= \{f \in L^p(\Omega,\mu) : f \restrict{A}=0 \text{ a.e.}\}}$. A limitation of this section is that spaces such as $C(K)$ do not possess any non-trivial band projections when $K$ is compact, Hausdorff, and connected~\cite[Example 5 on p. 63]{Schaefer1974}. Nevertheless, the variety of examples --  each operator $S$ in Section~\ref{section:applications} was actually a band projection -- suggest that the particular case of band projections is worthwhile to look into.

If $(e^{tA})_{t\geq 0}$ is a positive semigroup  on a Banach lattice $E$, then~\cite[Remark~B-III.2.15(a)]{Nagel1986} gives equivalent conditions for the spectral bound $\spb(A)$ to be a simple pole of the resolvent $\Res(\argument,A)$. In particular, it is shown that $\spb(A)$ is a simple pole if either $\ker(\spb(A)-A)$ contains a quasi-interior point or $\ker(\spb(A)-A')$ contains a strictly positive functional. In fact, as a consequence of~\cite[Theorem~5.2 and Corollary~3.3]{DanersGlueckKennedy2016b}, this is also true (under technical assumptions) if $(e^{tA})_{t\geq 0}$ is merely individually eventually positive. As a first result in this section, we prove a generalization of~\cite[Remark~B-III.2.15(a)]{Nagel1986}; see also~\cite[Proposition~7.9]{DanersGlueckKennedy2016b}.

\begin{proposition}
	\label{prop:characterization-simple-pole}
	Let $A$ be a closed, densely defined, and real operator on a complex Banach lattice $E$ such that $\lambda_0$ is a real spectral value of $A$ and a pole of the resolvent $\Res(\argument,A)$. Further, let $(S_n)$ and $(T_n)$ be two sequences of band projections on $E$ converging to the identity operator in the strong operator topology. If we have the convergence ${(\lambda-\lambda_0)\dist{S_n\Res(\lambda,A) T_nf}{E_+} \to 0}$ as $\lambda \downarrow \lambda_0$ for every $f \in E_+$ and for every $n\in \bbN$, then the following assertions are equivalent.
	\begin{enumerate}[ref=(\roman*)]
		\item\label{prop:characterization-simple-pole:pole} The spectral value $\lambda_0$ is a simple pole of the resolvent $\Res(\argument,A)$. 
		
		\item\label{prop:characterization-simple-pole:vector} For every positive and non-zero vector $x$ in the eigenspace $\ker(\lambda_0-A)$, there exists a positive functional $\psi$ in the dual eigenspace $\ker(\lambda_0-A')$ such that $\duality{\psi}{x}>0$. 
		
		\item\label{prop:characterization-simple-pole:functional} For every positive and non-zero functional $\psi$ in the eigenspace ${\ker(\lambda_0-A')}$, there exists a positive vector $x$ in the eigenspace $\ker(\lambda_0-A)$ such that $\linebreak{\duality{\psi}{x}>0}$.
	\end{enumerate}
	In particular, the spectral value $\lambda_0$ is a simple pole of the resolvent $\Res(\argument,A)$ if either $\ker(\lambda_0-A)$ contains a quasi-interior point or $\ker(\lambda_0-A')$ contains a strictly positive functional. 
	
	Moreover, if the equivalent assertions \ref{prop:characterization-simple-pole:pole}--\ref{prop:characterization-simple-pole:functional} hold, then the spectral projection corresponding to $\lambda_0$ is positive.
\end{proposition}

\begin{proof}
	We begin by noting that due to~\cite[Theorem~II.6.3]{Schaefer1974}, each of the conditions -- ${\ker(\lambda_0-A)}$ contains a quasi-interior point and $\ker(\lambda_0-A')$ contains a strictly positive functional -- imply the assertions \ref{prop:characterization-simple-pole:functional} and \ref{prop:characterization-simple-pole:vector} respectively. Next, we prove the equivalence of \ref{prop:characterization-simple-pole:pole} and \ref{prop:characterization-simple-pole:vector}. The equivalence of \ref{prop:characterization-simple-pole:pole} and \ref{prop:characterization-simple-pole:functional} can be proved in a parallel way. As usual, we will assume $\lambda_0=0$. 
	
	\Implies{prop:characterization-simple-pole:pole}{prop:characterization-simple-pole:vector}  Let $x$ be a positive and non-zero vector in $\ker A$ and denote the spectral projection associated to $0$ by $P$. Then by Proposition~\ref{prop:characterization-asymptotic-resolvent}, we have $S_n P T_n \geq 0$ for all $n \in \bbN$. Since the sequences $(S_n)$ and $(T_n)$ converge strongly to the identity operator, so in fact $P \geq 0$. Now pick any positive functional $\phi$ on $E$ such that $\duality{\phi}{x} >0$. Because $0$ is a simple pole of the resolvent $\Res(\argument,A)$, therefore, $\Ima P= \ker A$ and $\Ima P'=\ker A'$. As a result, $\psi:= P'\phi \in \ker A'$ and $\psi \geq 0$. Hence
	\[
		\duality{\psi}{x}=\duality{P'\phi}{x}=\duality{\phi}{Px}=\duality{\phi}{x}>0.
	\]
	
	\Implies{prop:characterization-simple-pole:vector}{prop:characterization-simple-pole:pole} Let $m \geq 2$ be the order of $0$ as a pole of the resolvent $\Res(\argument,A)$ and $U_{-m} \in \calL(E)$ be the coefficient of $\lambda^{-m}$ in the Laurent series expansion of $\Res(\argument,A)$ about $0$. Then there exists a positive non-zero vector $v \in E$ such that $x:=U_{-m}v >0$ (see Remark~\ref{rem:sufficient-existence-vectors}). Since $\Ima U_{-m} \subseteq \ker A$, we get
	\[
		\duality{\psi}{x}=\duality{\psi}{U_{-m}v}=\duality{U_{-m}'\psi}{v}=\duality{U_{-m+1}'A'\psi}{v}=0
	\]
	for every $0<\psi\in \ker A'$ (where the penultimate equality is obtained using $m\geq 2$ and~\cite[Theorem~VIII.8.2]{Yosida1980}), which is a contradiction.
	
	Lastly, suppose the equivalent assertions \ref{prop:characterization-simple-pole:pole}--\ref{prop:characterization-simple-pole:functional} hold. Then the proof of the implication ``\ref{prop:characterization-simple-pole:pole} $\Rightarrow$ \ref{prop:characterization-simple-pole:vector}'' shows that the spectral projection corresponding to $\lambda_0$ is positive.
\end{proof}

Let $\lambda_0$ be a real spectral value of a closed, densely defined, and real operator $A$ on a complex Banach lattice such that $\lambda_0$ is also a pole of the resolvent $\Res(\argument,A)$. If $A$ satisfies a \emph{domination condition}, then the resolvent $R(\argument,A)$ is individually eventually strongly positive (equivalently, negative) with respect to a quasi-interior point at $\lambda_0$ if and only if the corresponding spectral projection is strongly positive with respect to the same quasi-interior point~\cite[Theorem~4.4]{DanersGlueckKennedy2016b}. In fact, it was shown in~\cite[Theorem~4.1]{DanersGlueck2017}, that the necessary condition remains true even if the domination assumption is dropped. We obtain a version of this in the setting of this section. Observe that, in both Sections~\ref{section:sufficient-conditions-individual} and \ref{section:sufficient-conditions-uniform}, together with domination conditions, we implicitly (owing to Proposition~\ref{prop:sufficient-projection}) also assumed that the spectral projection is locally strongly positive while proving sufficient conditions for local eventual positivity (or negativity) of the resolvent. In contrast,  we do not presuppose any domination condition in the following theorem.

\begin{theorem}
	\label{thm:consequence-resolvent-strongly-positive-projection}
	Let $A$ be a closed, densely defined, and real operator on a complex Banach lattice $E$ and let $\lambda_0$ be a real spectral value of $A$ and a simple pole of the resolvent $\Res(\argument,A)$. Further, let $(S_n)$ and $(T_n)$ be sequences of band projections on $E$ converging strongly to the identity operator and let $u$ be a quasi-interior of $E$.
	
	Assume for each $n\in \bbN$ and for every positive non-zero vector $f \in E$, there exist $\lambda_1>\lambda_0$ and $c>0$ such that $ (\lambda_0,\lambda_1]\subseteq \rho(A)$  and ${(\lambda-\lambda_0) S_n \Res(\lambda,A) T_n f \geq c S_n u}$ for all $\lambda \in (\lambda_0,\lambda_1]$. Then the spectral projection associated to $\lambda_0$ is strongly positive.
\end{theorem}

Notice that the constant $c>0$ in our hypothesis is dependent on $n\in \bbN$ and $f> 0$ but is independent of $\lambda\in (\lambda_0,\lambda_1]$. This requirement, although relatively strong, allows us to obtain worthwhile results. As a matter of fact, this prerequisite is in line with the conclusion of the results in Sections~\ref{section:sufficient-conditions-individual} and \ref{section:sufficient-conditions-uniform}. Later on, we will discuss this condition in a little more detail, and eventually drop it in Theorems~\ref{thm:consequence-resolvent-spectral}, \ref{thm:consequence-semigroup-spectral}, and \ref{thm:consequence-uniform-convergence}.

Owing to~\cite[Proposition~3.1]{DanersGlueckKennedy2016b}, the strong positivity of the spectral projection is equivalent to:~the eigenspace $\ker(\lambda_0-A')$ contains a strictly positive functional and the eigenspace $\ker(\lambda_0-A)$ is spanned by a quasi-interior point. We will use this characterization to prove Theorem~\ref{thm:consequence-resolvent-strongly-positive-projection}. First, we give a sufficient condition for $\ker(\lambda_0-A)$ to be spanned by a quasi-interior point.

\begin{lemma}
	\label{lem:consequence-quasi-dimension}
	Let $E$ be a complex Banach lattice.
	\begin{enumerate}[ref=(\roman*)]
	\item\label{lem:consequence-quasi-dimension:projection} If $P$ is a positive projection on $E$, then $P$ maps quasi-interior points of $E$ to quasi-interior points of $\Ima P$. 
	
	\item\label{lem:consequence-quasi-dimension:quasi} Let $(S_n)$ be a sequence of band projections on $E$ converging strongly to the identity operator. If $x$ is a positive and non-zero vector in $E$, then it is a quasi-interior point of $E$ if and only if $S_n x$ is a quasi-interior point of $\Ima S_n$ for every $n \in \bbN$.
	
	\item\label{lem:consequence-quasi-dimension:dimension} Let $\lambda_0$ be a real spectral value of a closed, densely defined, and real operator $A$ on $E$. Moreover, let $\lambda_0$ be a simple pole of the resolvent $\Res(\argument,A)$ such that the associated spectral projection is positive. If every positive non-zero vector in the eigenspace $\ker(\lambda_0-A)$ is a quasi-interior point of $E$, then the dimension of the eigenspace $\ker(\lambda_0-A)$ is equal to one.
	\end{enumerate}
\end{lemma}

\begin{proof}
	\ref{lem:consequence-quasi-dimension:projection} Let $x$ be a quasi-interior point of $E$ and $0<\phi\in (\Ima P)'$. Then $P'\phi$ is a positive non-zero functional on $E$, which implies that $\duality{P'\phi}{x}>0$. Thus $\duality{\phi}{Px}>0$. Now the claim readily follows after appealing to~\cite[Theorem~II.6.3]{Schaefer1974}.
	
	\ref{lem:consequence-quasi-dimension:quasi} Firstly, fix $n\in\bbN$. Since $S_n$ is a positive projection, so $S_n$ maps quasi-interior points of $E$ to quasi-interior points of $\Ima S_n$  by~\ref{lem:consequence-quasi-dimension:projection}. Conversely, let $x\in E$ be positive and non-zero such that $S_n x$ is a quasi-interior point of $\Ima S_n$ for every $n \in \bbN$. To verify that $x$ is a quasi-interior of $E$, it is enough to show that $\duality{\phi}{x}>0$ for every $0<\phi \in E'$ due to~\cite[Theorem~II.6.3]{Schaefer1974}. To this end, let $\phi$ be a positive non-zero functional on $E$. Since $(S_n)$ converges strongly to the identity operator on $E$, thus $(S_n'\phi)$ converges to $\phi$ in the weak$^*$-topology. Hence, we can find $n\in \bbN$ such that $S_n'\phi \neq 0$. As a result, there exists $0<f\in E$ such that $\duality{S_n'\phi}{f}\neq 0$. Since $S_n$ is a projection, we obtain $\duality{S_n'\phi}{S_nf}\neq 0$. For this reason, $(S_n'\phi)\restrict{\Ima S_n}$ is a positive and non-zero functional on $\Ima S_n$. This, together with the fact that $S_n x$ is a quasi-interior point of $\Ima S_n$, allows us to again employ~\cite[Theorem~II.6.3]{Schaefer1974} to get $\duality{S_n'\phi}{S_n x} >0$. But $S_n$ and $S_n'$ are band projections, which in turn means $S_n x\leq x$ and $S_n'\phi\leq \phi$, so actually $\duality{\phi}{x}>0$.
		
	\ref{lem:consequence-quasi-dimension:dimension} Let $P$ denote the spectral projection associated to $\lambda_0$. Since $P$ is a positive projection, we first have that $\Ima P$ is itself a Banach lattice with respect to an equivalent norm~\cite[Proposition~III.11.5]{Schaefer1974} and secondly that, $P$ maps quasi-interior points of $E$ to quasi-interior points of $\Ima P$ by~\ref{lem:consequence-quasi-dimension:projection}. Because $\lambda_0$ is a simple pole of the resolvent we also have that ${\Ima P=\ker(\lambda_0-A)}$.  Consequently, if every positive non-zero vector in the eigenspace $\ker(\lambda_0-A)$ is a quasi-interior point of $E$, then it must also be a quasi-interior point of $\ker(\lambda_0-A)$. Therefore, the dimension of $\ker(\lambda_0-A)$ must be one by~\cite[Lemma~5.1]{Lotz1968}; alternatively see~\cite[Remark~5.9]{Glueck2018} for an English translation of the result.
\end{proof}

We are now in a position to illustrate why it is not sufficient to consider sequences of positive operators instead of band projections -- in consonance with the earlier sections. Consider the function $u(x)=1-x^2$ on $[-1,1]$ and for each $n\in \bbN$, define $S_nf = \one_{[-1+1/n,1-1/n]} f$ for all $f\in C[-1,1]$. Then $ u\in C[-1,1]$ and $(S_n)$ is a sequence of positive operators converging strongly to the identity operator. Moreover, $S_n u$ is a quasi-interior point of $\Ima S_n$ for each $n$, but $u$ is not a quasi-interior point of $C[-1,1]$ (as it vanishes on the boundary). On the other hand, because of Lemma~\ref{lem:consequence-quasi-dimension}\ref{lem:consequence-quasi-dimension:quasi}, if $S_n$ is a sequence of band projections on a complex Banach lattice converging strongly to identity, then a positive non-zero vector $u$ is a quasi-interior point of $E$ if and only if $S_n u$ is a quasi-interior point of $\Ima S_n$ for each $n\in\bbN$. 

\begin{proof}[Proof of Theorem~\ref{thm:consequence-resolvent-strongly-positive-projection}]
	Without loss of generality, suppose $\lambda_0=0$. Since $0$ is a simple pole of the resolvent $\Res(\argument, A)$, the corresponding spectral projection $P$ is positive (Proposition~\ref{prop:characterization-simple-pole}). In fact, we can also infer that $\lim_{\lambda\downarrow 0}{\lambda \Res(\lambda,A)= P}$ strongly, ${\Ima P=\ker A}$, and ${\Ima P'=\ker A'}$. Using the positivity of $P$ with the above convergence, we get
	\begin{equation}\label{eq:negative-resolvent}
		\lim_{\lambda \downarrow 0} \left(\lambda\Res(\lambda,A)f\right)^-  = (Pf)^- = 0  \, \, \text{for all } f \geq 0.
	\end{equation}
	Now using the positivity of $P$ (and $P'$) with the latter two equalities (of the corresponding range) gives the existence of a vector $0<x\in\ker A$ and a functional ${0<\psi<\ker A'}$. In order to prove our assertion, it is sufficient to show that $\ker A$ is one-dimensional, $x$ is a quasi-interior point of $E$, and the functional $\psi$ is strictly positive; see~\cite[Proposition~3.1]{DanersGlueckKennedy2016b}. 
	
	To this end, fix $n \in \bbN$. Then there exist $\lambda_1,c_1>0$ such that $(0,\lambda_1]\subseteq \rho(A)$ and $\lambda S_n\Res(\lambda,A)T_n x\geq c_1 S_n u$ for all $\lambda \in (0,\lambda_1]$. In addition, we also have that $\lambda\Res(\lambda,A)x=x$ for every $\lambda\in \rho(A)$. Combining these we obtain
	\begin{align*}
		S_n x &=\lambda S_n \Res(\lambda,A) T_n x + S_n \big(\lambda \Res(\lambda,A) (x- T_n x)\big)\\
			&\geq \lambda  S_n \Res(\lambda,A) T_n x - S_n \big(\lambda \Res(\lambda,A) (x- T_n x)\big)^-\\
			& \geq c_1 S_n u -S_n \big(\lambda \Res(\lambda,A) (x- T_n x)\big)^-
	\end{align*}
	for all $\lambda \in (0,\lambda_1]$. Note that as $T_n$ is a band projection, we have ${x- T_n x\geq 0}$. Therefore we can apply \eqref{eq:negative-resolvent} to the above inequality, and doing so gives the inequality ${S_n x \geq c_1 S_n u}$ (note that this can be done since $c_1$ does not depend on $\lambda$). As $u$ is a quasi-interior point of $E$ and $S_n$ is a positive projection, so the vector $S_n u$ is a quasi-interior point of $\Ima S_n$ (Lemma~\ref{lem:consequence-quasi-dimension}\ref{lem:consequence-quasi-dimension:projection}), and in turn, $S_n x$ is a quasi-interior point of $\Ima S_n$. This allows us to take advantage of Lemma~\ref{lem:consequence-quasi-dimension}\ref{lem:consequence-quasi-dimension:quasi} to conclude that $x$ is a quasi-interior point of $E$. In fact, we have shown that every positive non-zero element of $\ker A$ is a quasi-interior point of $E$. An application of Lemma~\ref{lem:consequence-quasi-dimension}\ref{lem:consequence-quasi-dimension:dimension} now implies that $\ker A$ is one-dimensional.
	
	Finally, it only remains to show that $\psi$ is a strictly positive functional. In order to see this, first observe that because $\psi$ is a positive non-zero functional and $S_n u$ is a quasi-interior point of $\Ima S_n$ for all $n\in \bbN$, we can argue as in the proof of Lemma~\ref{lem:consequence-quasi-dimension}\ref{lem:consequence-quasi-dimension:quasi}  to show that  $\duality{\psi}{S_n u}=\duality{S_n'\psi}{S_n u} >0$ for sufficiently large $n \in\bbN$. Choose one such $n \in \bbN$ and let $0<f\in E$. By our assumptions, there exist $\lambda_2, c_2>0$ such that $(0,\lambda_2]\subseteq \rho(A)$ and ${\lambda S_n \Res(\lambda,A)T_nf \geq c_2 S_n u}$ for every $\lambda \in (0,\lambda_2]$. In addition, $\lambda \Res(\lambda,A)'\psi=\psi$ for all $\lambda\in \rho(A)$. These together with the fact that $S_n'$ is a band projection (in particular, $\psi-S_n'\psi\geq 0$) yields
	\begin{align*}
		\duality{T_n'\psi}{f} & = \duality{\lambda T_n' \Res(\lambda,A)' S_n' \psi}{f}  + \duality{T_n' \lambda \Res(\lambda,A)' (\psi - S_n'\psi)}{f}\\
					& =\lambda\duality{\psi}{S_n \Res(\lambda,A) T_n f} + \duality{\psi-S_n'\psi}{\lambda\Res(\lambda,A)T_n f}\\
					&\geq c_2 \duality{\psi}{S_n u} - \duality{\psi-S_n'\psi}{\left(\lambda \Res(\lambda,A)T_n f\right)^{-}}
	\end{align*}
	for all $\lambda \in (0,\lambda_2]$. Furthermore, because $T_n$ is a band projection, we have ${T_n f \geq 0}$. Whence, in view of $\psi-S_n'\psi\geq 0$, the previous inequality, and \eqref{eq:negative-resolvent} we have that ${\duality{T_n'\psi}{f} \geq c_2\duality{\psi}{S_n u}>0}$.  Lastly, since $T_n'$ is also a band projection, so ${T_n'\psi\leq \psi}$ from which it follows that $\duality{\psi}{f} >0$, proving that $\psi$ is a strictly positive functional.
\end{proof}

By a change of signs, we have the following consequence of local eventual negativity of the resolvent:

\begin{corollary}
	\label{cor:consequence-resolvent-strongly-positive-projection}
	Let $A$ be a closed, densely defined, and real operator on a complex Banach lattice $E$ and $\lambda_0$ be a real spectral value of $A$ and a simple pole of the resolvent $\Res(\argument,A)$. Further, let $(S_n)$ and $(T_n)$ be sequences of band projections on $E$ converging strongly to the identity operator and let $u$ be a quasi-interior point of $E$.
	
	Assume for each $n\in \bbN$ and for each positive non-zero vector $f \in E$, there exist $\lambda_1<\lambda_0$ and $c>0$ so that $[\lambda_1,\lambda_0)\subseteq \rho(A)$ and ${(\lambda-\lambda_0) S_n \Res(\lambda,A) T_n f \geq c S_n u}$ for all $\lambda \in [\lambda_1,\lambda_0)$. Then the spectral projection associated to $\lambda_0$ is strongly positive.
\end{corollary}

We now turn our attention to implications of local eventual positivity of the semigroup. Under certain spectral and smoothing assumptions, individual eventual positivity of the semigroup is equivalent to strong convergence of the semigroup operators to a strongly positive projection~\cite[Theorem~5.2]{DanersGlueckKennedy2016b}. Furthermore, it was shown in~\cite[Theorem~8.3]{DanersGlueckKennedy2016b} that (under technical assumptions) if the semigroup is \emph{individually asymptotically positive}, then the spectral bound of the generator is a dominant spectral value and a simple pole of the resolvent. Moreover, the operators $e^{t(A-\spb(A))}$ converge strongly to a positive projection as $t\to\infty$. We show that these consequences hold in our setting as well. Our arguments closely follow the proof of~\cite[Theorem~8.3]{DanersGlueckKennedy2016b}.

\begin{theorem}
	\label{thm:consequence-strong-convergence}
	Let $(e^{tA})_{t\geq 0}$ be a real $C_0$-semigroup on a complex Banach lattice $E$ such that the spectral bound $\spb(A)>-\infty$. Assume that the peripheral spectrum $\perSpec(A)$ is non-empty, finite, and consists only of poles of the resolvent $\Res(\argument,A)$ and that the rescaled semigroup $(e^{t(A-\spb(A))})_{t\geq 0}$ is bounded.
	
	Let $(S_n)$ and $(T_n)$ be sequences of band projections on $E$ converging strongly to the identity operator such that for every $f \geq 0$ and $n\in \bbN$, we have
	\[
		\dist{S_n e^{t(A-\spb(A))}T_n f}{E_+}\to 0 \text{ as } t\to \infty.
	\]
	
	Then the spectral bound $\spb(A)$ is a dominant spectral value of $A$, a simple pole of the resolvent $\Res(\argument,A)$, and the rescaled semigroup $(e^{t(A-\spb(A))})_{t\geq 0}$ converges strongly to the spectral projection $P$ associated to $\spb(A)$. Furthermore, $P$ is positive.
\end{theorem}

\begin{proof}
	Let $\perP$ denote the spectral projection associated to $\perSpec(A)$ and assume without loss of generality that $\spb(A)=0$. Since the semigroup is bounded, every pole of the resolvent $\Res(\argument,A)$ on the imaginary axis  is a simple pole and so $\perSpec(A)$ consists solely of simple poles of the resolvent. Therefore by~\cite[Proposition~2.3]{DanersGlueckKennedy2016a}, there exists a sequence of positive integers $t_m \to \infty$ such that $e^{t_m A}\perP \to \perP$ in the strong operator topology. Further, using boundedness of the semigroup together with the fact that both, the point spectrum of $A'\restrict{\ker \perP}$ and the spectrum of $A\restrict{\ker \perP}$ do not intersect the imaginary axis, we obtain as a consequence of the ABLV theorem~\cite[Theorem~V.2.21]{EngelNagel2000} that the semigroup converges strongly to $0$ on $\ker \perP$. In addition to this, if we are able to show that $\perP$ is positive and ${\perSpec(A)=\{0\}}$, it would mean that the spectral projection associated to $0$ is $\perP\geq 0$ and accordingly $e^{tA}\perP=\perP$. As a result, the theorem would follow.
	
	 To this end, fix $n,m \in \mathbb{N}$ and $t\geq 0$. Then
	 \[
	 	S_n e^{(t+t_m)A}\perP T_n = S_n e^{(t+t_m)A} T_n - S_ne^{(t+t_m)A} (I-\perP)T_n.
	 \]
	  Since for each $f \in E_+$, we have $\lim_{m\to \infty}\dist{S_n e^{(t+t_m)A} T_n  f}{E_+} =0$, therefore
	 \[
	 	\lim_{m\to\infty}\dist{S_n e^{(t+t_m)A}\perP T_n  f}{E_+} \leq \lim_{m\to\infty} \norm{S_ne^{(t+t_m)A} (I-\perP)T_n f}.
	 \]
	It follows that $S_n e^{tA} \perP T_n \geq 0$ for every $n \in \bbN$ and ${t\geq 0}$. Consequently, $\perP$ is a positive projection and the semigroup $\left(e^{tA}\restrict{\Ima \perP}\right)_{t\geq 0}$ is positive. The former along with~\cite[Proposition~III.11.5]{Schaefer1974} implies that $\Ima \perP$ is a complex Banach lattice under an equivalent norm. 
	
	Finally, note that $\spec\left(A \restrict{\Ima \perP}\right) = \perSpec(A) \neq \emptyset$ and so
		\[
			0=\spb(A)=\spb\left(A\restrict{\Ima \perP}\right) \in \spec\left(A\restrict{\Ima \perP}\right)=\perSpec\left(A\restrict{\Ima \perP}\right) =\perSpec(A);
		\]
	here we have used that if the spectrum of the generator of a positive semigroup is non-empty, then it contains the spectral bound of the generator~\cite[Corollary~C-III.1.4]{Nagel1986}. In addition, because the semigroup $\left(e^{tA}\restrict{\Ima \perP}\right)_{t\geq 0}$ is positive, bounded, and the spectral bound $\spb\left(A\restrict{\Ima \perP}\right)=0$, we infer that the peripheral spectrum $\perSpec\left(A\restrict{\Ima \perP}\right)$ is cyclic (this is proved, for instance in~\cite[Proposition~C-III.2.9 and Theorem~C-III.2.10]{Nagel1986}). In other words, if $i\beta \in \perSpec\left(A\restrict{\Ima \perP}\right)$, then we must also have ${ik\beta \in \perSpec\left(A\restrict{\Ima \perP}\right)}$ for every $k \in \bbZ$. On the other hand, $\perSpec(A)$ is assumed to be finite. The cyclicity and finiteness of the peripheral spectrum can simultaneously occur if and only if ${\perSpec(A) =\{0\}}$.
\end{proof}

Employing Theorem~\ref{thm:consequence-strong-convergence} and with similar arguments as in the proof of Theorem~\ref{thm:consequence-resolvent-strongly-positive-projection}, we can obtain strong positivity of the spectral projection as a consequence of local eventual positivity of the semigroup: 

\begin{corollary}
	\label{cor:consequence-semigroup-strongly-positive-projection}
	Let $(e^{tA})_{t\geq 0}$ be a real $C_0$-semigroup on a complex Banach lattice $E$ such that the spectral bound $\spb(A)>-\infty$. Assume that the peripheral spectrum $\perSpec(A)$ is non-empty, finite, and consists only of poles of the resolvent $\Res(\argument,A)$ and that the rescaled semigroup $(e^{t(A-\spb(A))})_{t\geq 0}$ is bounded.
	
	Let $(S_n)$ and $(T_n)$ be sequences of band projections on $E$ converging strongly to the identity operator and let $u$ be a quasi-interior point of $E$. Moreover, suppose for every $n \in \bbN$ and every positive non-zero vector $f \in E$, there exist $t_0 \geq 0$ and $c >0$ such that $S_n e^{tA} T_n f \geq c S_n u$ for all $t\geq t_0$.
	
	Then the spectral bound of $A$ is a pole of the resolvent $\Res(\argument,A)$ and the associated spectral projection is strongly positive.
\end{corollary}

\begin{proof}
	As usual we assume $\spb(A)=0$. Due to Theorem~\ref{thm:consequence-strong-convergence}, we have that ${0\in\perSpec(A)}$  (hence is a pole of the resolvent $\Res(\argument,A)$) and also the strong convergence of  the operators $e^{tA}$ to $P\geq 0$  as $ t\to \infty$; here $P$ denotes the spectral projection associated to $0$. Therefore
	\begin{equation}\label{eq:negative-semigroup}
		\lim_{t \to \infty} \left(e^{tA}f\right)^- =  (Pf)^- = 0  \, \, \text{for every } f \geq 0.
	\end{equation}
	
	Furthermore, Theorem~\ref{thm:consequence-strong-convergence} also asserts that $0$ is a simple pole of the resolvent $\Res(\argument,A)$. Thus $\Ima P=\ker A$ and $\Ima P'=\ker A'$. The positivity of $P$ also gives the positivity of $P'$ and so we can pick a vector $0<x\in\ker A$ and a functional $0<\psi\in \ker A'$. In view of~\cite[Proposition~3.1]{DanersGlueckKennedy2016b}, it is  enough to show that $\ker A$ is one-dimensional, $x$ is a quasi-interior point of $E$, and that $\psi$ is a strictly positive functional.
	
	In order to prove the above, we argue in a manner analogous to the proof of Theorem~\ref{thm:consequence-resolvent-strongly-positive-projection}. Start by fixing $n \in \bbN$. By assumption, there exist $t_1 \geq 0$ and $c_1>0$ such that ${S_n e^{tA} T_n x \geq c_1 S_n u}$ for all $t \geq t_1$. Additionally, $e^{tA}x=x$ for every $t\geq 0$. Hence
	\begin{align*}
		S_n x &= S_n e^{tA} T_n x + S_n \big(e^{tA} (x- T_n x)\big)\\
			& \geq S_n e^{tA} T_n x - S_n \big(e^{tA} (x- T_n x)\big)^-\\
			& \geq c_1 S_n u -S_n \big(e^{tA} (x- T_n x)\big)^-
	\end{align*}
	for all $t \geq t_1$. Keeping in mind $x-T_n x\geq 0$ and applying \eqref{eq:negative-semigroup}, we obtain ${S_n x\geq c_1 S_n u}$. Arguing exactly as in the proof of Theorem~\ref{thm:consequence-resolvent-strongly-positive-projection} and employing Lemma~\ref{lem:consequence-quasi-dimension}, we can show that $x$ is a quasi-interior point of $E$ and that $\ker A$ is one-dimensional.
	
	Finally, let $0<f \in E$. As in the proof of Theorem~\ref{thm:consequence-resolvent-strongly-positive-projection}, we pick $n \in \bbN$ such that $\duality{\psi}{S_n u}>0$. Once again, our assumption implies the existence of $t_2 \geq 0$ and $c_2>0$ such that $S_n e^{tA} T_n f \geq c_2 S_n u$ for all ${t \geq t_2}$. Also, $(e^{tA})'\psi =\psi $  for all $t\geq 0$. As a result,
	\begin{align*}
		\duality{T_n'\psi}{f} & = \duality{T_n' (e^{tA})' S_n' \psi}{f}  + \duality{T_n' (e^{tA})' (\psi - S_n'\psi)}{f}\\
					& =\duality{\psi}{S_n e^{tA} T_n f} + \duality{\psi-S_n'\psi}{e^{tA}T_n f}\\
					&\geq c_2\duality{\psi}{S_n u} - \duality{\psi-S_n'\psi}{\left(e^{tA}T_n f\right)^{-}}
	\end{align*}
	for all $t \geq t_2$. Employing \eqref{eq:negative-semigroup} a second time, and repeating the final arguments of Theorem~\ref{thm:consequence-resolvent-strongly-positive-projection} yields that $\psi$ is strictly positive.
\end{proof}

We point out that the situation for resolvents mentioned before Theorem~\ref{thm:consequence-resolvent-strongly-positive-projection} holds for semigroups as well~\cite[Theorem~5.2]{DanersGlueckKennedy2016b}, namely:~Let $(e^{tA})_{t\geq 0}$ be a real $C_0$-semigroup on a complex Banach lattice $E$ such that the spectrum of $A$ is non-empty and the peripheral spectrum of $A$ is finite and consists of only poles of the resolvent of $A$. Assume the \emph{smoothing condition} is satisfied, that is, there exists a quasi-interior point $u$ of $E$ and $t_0\geq 0$ such that $e^{t_0A}E \subseteq E_u$. Then the semigroup is individually eventually strongly positive with respect to $u$ if and only if the rescaled semigroup $(e^{t(A-\spb(A))})_{t\geq 0}$ is bounded, $\spb(A)$ is a dominant spectral value and the corresponding spectral projection $P$ satisfies $P\gg_u 0$. Additionally, in~\cite[Theorem~5.1]{DanersGlueck2017}, it was shown that one can prove $P\gg_u 0$ even without the smoothing assumption. In the same way, when we proved a sufficient condition for uniform local eventual positivity of the semigroup in Theorem~\ref{thm:sufficient-uniform-semigroup}, along with a smoothing condition, we tacitly assumed (as a  result of Proposition~\ref{prop:sufficient-projection}) that the spectral projection is locally strongly positive; see also Remark~\ref{rem:sufficient-individual}(a). Conversely, we did not impose any smoothing assumptions  in Corollary~\ref{cor:consequence-semigroup-strongly-positive-projection}.

In Theorem~\ref{thm:consequence-resolvent-strongly-positive-projection}, we looked at a situation where strong local eventual positivity of the resolvent at a spectral value implies that the spectral projection associated to a real spectral value is strongly positive. However, as pointed out previously, the assumptions there were quite strong. Specifically, in the condition ${(\lambda-\lambda_0)S_n\Res(\lambda,A)T_nf\geq cS_n u}$, we assumed that the constant $c>0$ is the same for each $\lambda\in (\lambda_0,\lambda_1]$ and depends only on $n\in \bbN$ and $f\geq 0$. The proof does not work if $c$ also depends on $\lambda$. In particular, we cannot take the limit $\lambda \to 0$ to conclude $S_n x \geq c_1 S_n u$. A similar hypothesis was considered in both Corollaries~\ref{cor:consequence-resolvent-strongly-positive-projection} and \ref{cor:consequence-semigroup-strongly-positive-projection}. However, imposing this condition was not too surprising, since (as indicated earlier)  the results in Sections~\ref{section:sufficient-conditions-individual} and \ref{section:sufficient-conditions-uniform} had similar conclusions. For the rest of this section, we will drop this seemingly strong assumption. As a trade-off though, we only deal with a single sequence of band projections $(S_n)$ and are unable to show strong positivity of the spectral projection.

\begin{theorem}
	\label{thm:consequence-resolvent-spectral}
		Let $A$ be a closed, densely defined, and real operator on a complex Banach lattice $E$ such that $\lambda_0$ is a real spectral value of $A$ and a pole of the resolvent $\Res(\argument,A)$. Suppose $u$ is a quasi-interior point of $E$ and $(S_n)$ is a sequence of band projections on $E$ converging strongly to the identity operator. Furthermore, assume at least one of the following assumptions holds:
		\begin{enumerate}[\upshape (a)]
			\item For each $n \in \bbN$, the family $S_n \Res(\argument, A)$ is individually eventually strongly positive with respect to $S_n u$ at $\lambda_0$.
			
			\item For each $n \in \bbN$, the family $S_n \Res(\argument, A)$ is individually eventually strongly negative with respect to $S_n u$ at $\lambda_0$.
		\end{enumerate}
		
	Then the eigenspace $\ker(\lambda_0-A)$ is one-dimensional and spanned by a quasi-interior point of $E$. Moreover, $\lambda_0$ is a simple pole of the resolvent $\Res(\argument,A)$ and the associated spectral projection is positive.
\end{theorem}

\begin{proof}
	There is no loss of generality in assuming $\lambda_0=0$. Since the sequence $(S_n)$ converges to the identity operator in the strong operator topology, therefore by Remark~\ref{rem:sufficient-existence-vectors}, $\ker A$ contains a positive non-zero element, say $x$. 
	
	If (a) holds, then for each $n\in\bbN$, there exists $\lambda>0$ and such that
	\[
		 S_n \Res(\lambda,A)x \gg_{S_n u} 0;
	\]
	and if (b) holds, then for each $n\in\bbN$, there exists $\lambda<0$ and such that
	\[
		S_n \Res(\lambda,A)x \ll_{S_n u} 0.
	\]
	In either case, 
	\[
		S_n x= \lambda S_n \Res(\lambda,A)x \gg_{S_n u} 0.
	\]
	In particular, $S_n x$ is a quasi-interior point of $\Ima S_n$ for every $n\in \bbN$; here we have used that if $u$ is a quasi-interior point of $E$, then $S_n u$ is a quasi-interior point of $\Ima S_n$ for each $n\in\bbN$ (Lemma~\ref{lem:consequence-quasi-dimension}\ref{lem:consequence-quasi-dimension:quasi}). Appealing again to Lemma~\ref{lem:consequence-quasi-dimension}\ref{lem:consequence-quasi-dimension:quasi}, we deduce that the eigenvector $x$ is a quasi-interior point of $E$. Thus by Proposition~\ref{prop:characterization-simple-pole}, the spectral value $0$ is a simple pole of the resolvent $\Res(\argument,A)$ and the associated spectral projection is positive. Moreover, notice that we have proved that every positive non-zero element of $\ker A$ is a quasi-interior point of $E$. Hence, $\dim \ker A=1$ using Lemma~\ref{lem:consequence-quasi-dimension}\ref{lem:consequence-quasi-dimension:dimension}.
\end{proof}

An analogous result can be proved for the semigroups:

\begin{theorem}
	\label{thm:consequence-semigroup-spectral}
	Let $(e^{tA})_{t\geq 0}$ be a real $C_0$-semigroup on a complex Banach lattice $E$ such that $A$ has non-empty spectrum and the spectral bound $\spb(A)$ is a pole of the resolvent $\Res(\argument,A)$. Suppose $u$ is a quasi-interior point of $E$ and $(S_n)$ is a sequence of band projections on $E$ converging strongly to the identity operator.
	
	If for each $n \in \bbN$, the family $(S_n e^{tA})_{t\geq 0}$ is individually eventually strongly positive with respect to $S_n u$, then the eigenspace $\ker(\spb(A)-A)$ is one-dimensional and spanned by a quasi-interior point of $E$. Moreover, the spectral bound $\spb(A)$ is a simple pole of the resolvent $\Res(\argument,A)$ and the corresponding spectral projection is positive.
\end{theorem}

\begin{proof}
	We may assume $\spb(A)=0$. Then by Corollary~\ref{cor:laplace-representation} and Remark~\ref{rem:sufficient-existence-vectors}, there exists a positive non-zero vector $x\in \ker A$. By assumption, for each $n\in\bbN$, there exists $t>0$ such that ${S_n x= S_n e^{tA} x \gg_{S_n u} 0}$. The result can now be argued exactly as in the proof of Theorem~\ref{thm:consequence-resolvent-spectral}.
\end{proof}

Due to Remark~\ref{rem:sufficient-existence-vectors}, there exists a positive non-zero functional in both the eigenspaces ${\ker(\lambda_0-A')}$ and ${\ker(\spb(A)-A')}$ in Theorem~\ref{thm:consequence-resolvent-spectral} and Theorem~\ref{thm:consequence-semigroup-spectral} respectively. However, it is unclear whether this functional is strictly positive. If one is able to show that it is indeed strictly positive, then we can conclude that the corresponding spectral projection is strongly positive with the aid of~\cite[Proposition~3.1]{DanersGlueckKennedy2016b}. 

In Theorem~\ref{thm:consequence-strong-convergence}, we proved that (under technical assumptions) the local \emph{asymptotic} positivity of the semigroup implies that the semigroup converges (after an appropriate rescaling) strongly. We now turn our attention to operator norm convergence. In Theorem~\ref{thm:consequence-uniform-convergence}, we show that if we have eventual strong positivity of $(S_n e^{tA})_{t\geq 0}$ and $(e^{tA})_{t\geq 0}$ is norm continuous at infinity, then the convergence in Theorem~\ref{thm:consequence-strong-convergence} is actually in the operator norm topology. In contrast, we do not need to a priori assume the boundedness of $(e^{t(A-\spb(A))})_{t\geq 0}$.  Note that in the proof of Theorem~\ref{thm:sufficient-uniform-semigroup}, we showed that its assumptions ensure that the (rescaled) semigroup converges uniformly.

We briefly mention that uniform convergence of the (rescaled) semigroup was characterized in~\cite[Proposition~2.3]{Webb1987} and~\cite[Theorems~2.7 and 3.3]{Thieme1998}, while a refinement for positive semigroups was given in~\cite[Theorem~3.4]{Thieme1998}. Motivated by this, necessary and sufficient conditions for convergence of uniformly eventually positive semigroups (after appropriate rescaling) in the operator norm were recently studied in~\cite[Section~5]{AroraGlueck2020}.

\begin{theorem}
	\label{thm:consequence-uniform-convergence}
	On a complex Banach lattice $E$, let $(e^{tA})_{t\geq 0}$ be a real $C_0$-semigroup on $E$ that is norm continuous at infinity. Assume $\spb(A)\in\spec(A)$ and that the peripheral spectrum $\perSpec(A)$ consists only of poles of the resolvent $\Res(\argument,A)$. Suppose $u$ is a quasi-interior point of $E$ and $(S_n)$ is a sequence of band projections on $E$ converging strongly to the identity operator.
	
	If for each $n \in \bbN$, the family $(S_n e^{tA})_{t\geq 0}$ is individually eventually strongly positive with respect to $S_n u$, then the rescaled semigroup $(e^{t(A-\spb(A))})_{t\geq 0}$ converges to the spectral projection associated to $\spb(A)$ in the operator norm.
\end{theorem}

We outsource an important part of the argument to the following lemma -- since it is true under weaker assumptions. It gives information about the order of any pole of the resolvent in the peripheral spectrum $\perSpec(A)$, when one a priori knows that $\spb(A)$ is a pole of the resolvent (cf.~\cite[Corollary~C-III-1.5]{Nagel1986} and~\cite[Theorem~7.7(ii)]{DanersGlueckKennedy2016a}).

\begin{lemma}
	\label{lem:consequence-all-poles-simple}
	Let $(e^{tA})_{t\geq 0}$ be a real $C_0$-semigroup on a complex Banach lattice $E$ such that spectrum of $A$ is non-empty and $\spb(A)$ is a pole of the resolvent $\Res(\argument,A)$. Furthermore, let $(S_n)$ and $(T_n)$ be sequences of band projections on $E$ converging strongly to the identity operator. If for each $n \in \bbN$, the family $(S_n e^{tA}T_n)_{t\geq 0}$ is individually eventually positive, then the order of every pole of the resolvent $\Res(\argument,A)$ in $\perSpec(A)$ is at most the pole order of $\spb(A)$.
\end{lemma}

\begin{proof}
	We assume that $\spb(A)=0$  and define ${Q_t:=S_n e^{tA} T_n}$ for $t\geq 0$ and fixed $n \in \bbN$. By assumption, there exists $t_0 \geq 0$ such that $Q_t f^{+} \geq 0$ and $Q_t f^{-} \geq 0$ for all $t \geq t_0$. Let $f \in E_{\bbR}$ and $i\beta \in \perSpec(A)$ be a pole of the resolvent $\Res(\argument,A)$. Then for every $\alpha >0$, we have
	\begin{align*}
		\modulus{\int_0^{\tau} e^{-(\alpha+i\beta) t} Q_t f\, dt} &\leq \int_0^{\tau} e^{-\alpha t} \modulus{Q_t f}\, dt\\
											&\leq \int_0^{\tau} e^{-\alpha t} Q_t \modulus{f}\, dt + \int_0^{t_0} e^{-\alpha t} \left( \modulus{Q_t f}- Q_t \modulus{f}\right)\, dt
	\end{align*}
	for all $\tau \geq t_0$. We can now use our Laplace representation of the resolvent (Proposition~\ref{prop:laplace-representation}), which gives
	\[
		\modulus{S_n\Res(\lambda,A)T_n f} \leq S_n \Res(\alpha,A)T_n \modulus{f} +\int_0^{t_0} e^{-\alpha t} \left( \modulus{Q_t f}- Q_t\modulus{f}\right)\, dt,
	\]
	where $\lambda = \alpha + i \beta$ and $\alpha>0$. This implies
	\[
		\modulus{(\lambda-i\beta)^k S_n \Res(\lambda,A)T_n f} \leq \alpha^k S_n\Res(\alpha,A)T_n \modulus{f} + \alpha^k \int_0^{t_0} e^{-\alpha t} \left( \modulus{Q_t f}- Q_t\modulus{f}\right)\, dt
	\]
	for all $k\in \bbN$. If $0$ is a pole of order, say $m\in\bbN$, then using convergence of the sequences $(S_n)$ and $(T_n)$ to the identity operator, we may conclude the convergence ${(\lambda-i\beta)^k \Res(\lambda,A) f \to 0}$ as $\lambda \downarrow i\beta$ for all $k \geq m$ and hence the order of $i\beta$ as a pole of the resolvent $\Res(\argument,A)$ is at most $m$.
\end{proof}

\begin{proof}[Proof of Theorem~\ref{thm:consequence-uniform-convergence}]
	Without loss of generality, assume $\spb(A)=0$. Since the semigroup is norm continuous at infinity, so by~\cite[Theorem~1.9]{MartinezMazon1996}, we have compactness of the set ${\{\lambda\in \spec(A) : \re \lambda \geq -\epsilon\}}$ for some $\epsilon>0$. Therefore $\perSpec(A)$ must be compact as well. This implies that it must be finite and isolated from the rest of the spectrum as it consists only of poles. Let $\perSpec(A)=\{i\beta_1,i\beta_2,\ldots,i\beta_k\}$. 
	
	Now by Theorem~\ref{thm:consequence-semigroup-spectral}, the spectral value $0$ is a simple pole of the resolvent $\Res(\argument,A)$. Therefore, every pole of $\perSpec(A)$ is also simple according to Lemma~\ref{lem:consequence-all-poles-simple}. Hence if $\perP$ is the spectral projection associated to $\perSpec(A)$, then we have the decomposition ${\Ima \perP= \oplus_{j=1}^k \ker(i\beta_j - A)}$ (see~\cite[Theorem~VIII.8.3]{Yosida1980}). This shows that the semigroup is bounded on $\Ima \perP$. Moreover, compactness of the set ${\{\lambda\in \spec(A) : \re \lambda \geq -\epsilon\}}$ also implies that ${\spb\left(A\restrict{\ker \perP}\right)<0}$. 
	
	We claim that the growth bound $\gbd\left(A\restrict{\ker \perP}\right)$ of the restricted semigroup $\left(e^{tA}\restrict{\ker \perP}\right)_{t\geq 0}$ is negative as well. First of all, 
	\[
		\gbd\left(A\restrict{\ker \perP}\right) \leq \gbd(A)=\spb(A)=0,
	\]
	where the second last equality is true because growth bound of the semigroup coincides with the spectral bound of the generator for a semigroup that is norm continuous at infinity~\cite[Corollary~1.4]{MartinezMazon1996}. Let if possible, $\gbd\left(A\restrict{\ker \perP}\right)=0$. Then by definition, the semigroup $\left(e^{tA}\restrict{\ker \perP}\right)_{t\geq 0}$ would also be norm continuous at infinity. This in turn implies that ${\gbd\left(A\restrict{\ker \perP}\right) = \spb\left(A\restrict{\ker \perP}\right)<0}$, which is a contradiction and our claim follows. As a result, $(e^{tA})_{t\geq 0}$ converges to $0$ uniformly on $\ker \perP$ as $t\to\infty$, and in turn is bounded on $\ker \perP$. Consequently, the semigroup is bounded on ${E= \Ima \perP \oplus \ker \perP}$, as well.
	
	Finally, in Theorem~\ref{thm:consequence-strong-convergence} we showed that $0$ is a dominant spectral value and the associated spectral projection is $\perP$. For this reason, $e^{tA}$ acts as the identity operator on $\perP$. This, together with the uniform convergence to $0$ of the semigroup on $\ker \perP$, results in the convergence $e^{tA} \to \perP$ in the operator norm as ${t\to \infty}$.
\end{proof}

\section*{Acknowledgement}
I am indebted to Jochen Gl\"uck for numerous discussions and for contributing several ideas.  I am also grateful to him for his insightful comments and suggestions which helped me to further improve the paper.

My thanks also go to Ralph Chill for reading various drafts of the paper in detail and providing valuable remarks. In particular, his advice considerably generalized the results in Sections~\ref{section:sufficient-conditions-individual}, \ref{section:sufficient-conditions-uniform}, and \ref{section:consequences-i}; whereas the first drafts of the paper were restricted to band projections instead of positive operators -- and hence not applicable on the space of continuous functions.

I also thank Davide Addona for pointing out a mistake in an earlier version of the paper and Federica Gregorio for conveying the same.

Some results of this paper, namely Theorems~\ref{thm:local-anti-maximum} and \ref{thm:consequence-resolvent-strongly-positive-projection} and Corollaries~\ref{cor:local-maximum}, \ref{cor:consequence-resolvent-strongly-positive-projection}, and \ref{cor:consequence-semigroup-strongly-positive-projection} were initiated during a very pleasant stay of the author at Universit\"at Passau.

The author was financially supported by Deutscher Akademischer Austauschdienst (Forschung\-sstipendium-Promotion in Deutschland).


\end{document}